\documentclass{amsart}
\usepackage{amssymb,latexsym}
\usepackage{amsmath}
\usepackage{graphicx}
\usepackage{amscd}
\usepackage{color}
\usepackage{enumerate}
\newenvironment{enumeratei}{\begin{enumerate}[\upshape (i)]}{\end{enumerate}}
\numberwithin{equation}{section}
\theoremstyle{plain}
 \newtheorem{theorem}{Theorem}[section]
 \newtheorem{lemma}[theorem]{Lemma}

 \newtheorem{corollary}[theorem]{Corollary}
 
\theoremstyle{definition}
 \newtheorem{definition}[theorem]{Definition}

\newcommand \colgammacon{\textup{(C1)}}
\newcommand \colcongamma{\textup{(C2)}}
\newcommand \pcrely {principal congruence representability}
\newcommand \labc {\textup{lab}}
\renewcommand\phi{\varphi}

\renewcommand\rho{\varrho}
\newcommand\length{\textup{length}}
\newcommand\inp{\mathfrak p}
\newcommand\inq{\mathfrak q}
\newcommand\inr{\mathfrak r}
\newcommand \Con  {\textup{Con}}
\newcommand \Prime  {\textup{Prime}}

\newcommand \erep  {\textup{erep}}
\newcommand \srep  {\textup{SRep}}
\newcommand \restrict [2] {{#1}\rceil_{\kern-1pt #2}}
\newcommand\ideal[1]{\mathord\downarrow #1}
\newcommand\filter[1]{\mathord\uparrow #1}
\newcommand\filtrof[2]{\mathord\uparrow_{\kern-2pt #1}\kern 1pt #2}
\newcommand \tuple [1] {\langle #1\rangle}
\newcommand \pair [2] {\tuple{#1,#2}}
\newcommand \ppair [2] {\tuple{#1;#2}}
\newcommand \tripl [3] {\tuple{#1,#2,#3}}

\newcommand \kvint [5] {\tuple{#1,#2,#3,#4,#5}}
\newcommand \tbf [1] {\textbf{#1}} 
\newcommand \set[1] {\{#1\}}
\newcommand \jj {\vee}
\newcommand \Aut {\textup{Aut}}
\newcommand \Princ {\textup{Princ}}
\newcommand \Intv {\textup{Intv}}
\newcommand \con {\textup{con}}

\newcommand \balpha {\boldsymbol\alpha}
\newcommand \bbeta {\boldsymbol\beta}
\newcommand \semmi [1]{}
\newcommand \preogen {\textup{quo}}
\DeclareMathOperator{\Quo}{Quo}

\newcommand\ppdn {\overset{\textup{p-dn}}\rightarrow}
\newcommand\ppup {\overset{\textup{p-up}}\rightarrow}

\newcommand \cast {C^\ast}
\newcommand \dker {\mathop{\vec{\textup{Ker}}}}
\newcommand \what [1]{\widehat #1}
\newcommand \natc {\gamma_{\textup{nat}}}
\newcommand \who [1]{\widetilde #1}
\newcommand \whg {\what\gamma}
\newcommand \late {\what 1}
\newcommand \latn {\what 0}
\newcommand \pe {\pmb{\boldsymbol 1}}

\renewcommand \epsilon{\varepsilon}
\newcommand \erel {\epsilon}
\newcommand \chc [1] {\mathsf C_{#1}}
\begin{document}
\title[Fully principal congruence representable distributive lattices]
{Characterizing fully principal congruence representable distributive lattices}

\author[G.\ Cz\'edli]{G\'abor Cz\'edli}
\email{czedli@math.u-szeged.hu}
\urladdr{http://www.math.u-szeged.hu/\textasciitilde{}czedli/}
\address{University of Szeged\\ Bolyai Institute\\Szeged,
Aradi v\'ertan\'uk tere 1\\ Hungary 6720}

\thanks{This research was supported by
NFSR of Hungary (OTKA), grant number K 115518}
\begin{abstract} Motivated by a recent  paper of G.\ Gr\"atzer, a finite distributive lattice $D$ is \emph{fully principal congruence representable} if 
 for every subset $Q$ of $D$ containing $0$, $1$, and the set $J(D)$ of nonzero join-irreducible elements of $D$, there exists a finite lattice $L$ and an isomorphism from the congruence lattice of $L$ onto $D$ such that $Q$ corresponds to the set of principal congruences of $L$ under this isomorphism. Based on earlier results of G.\ Gr\"atzer, H.\ Lakser, and the present author, we prove that a finite distributive lattice $D$ is  fully principal congruence representable if and only if it is planar and it has at most one join-reducible coatom. Furthermore, even the automorphism group of $L$ can arbitrarily be stipulated in this case. 
Also, we generalize a recent result of G.\ Gr\"atzer on principal congruence representable subsets of a distributive lattice whose top element is join-irreducible by proving that the automorphism group of the lattice we construct can be  arbitrary.
\end{abstract}

\subjclass {06B10{{\hfill \color{red} June 11, 2017\color{black}}}}
\keywords{Distributive lattice, principal lattice congruence, congruence lattice,  chain-representability, simultaneous representation, automorphism group.}

\maketitle
\section{Introduction and our main goal}
Unless otherwise specified explicitly, 
all lattices in this paper are assumed to be finite, even if this is not repeated all the time. 
For a finite lattice $L$,  $J(L)$ denotes the ordered set of nonzero join-irreducible elements of $L$, $J_0(L)$ stands for $J(L)\cup\set0$, and we let $J^+(L)=J(L)\cup\set{0,1}$. 
Also, $\Princ(L)$ denotes the \emph{ordered set of all principal congruences} of $L$;  it is a subset of the \emph{congruence lattice} $\Con(L)$ of $L$ and a superset of $J^+(\Con(L))$. It is well known that 
$\Con(L)$ is distributive. These facts motivate the following concept, which is due to Gr\"atzer~\cite{ggwith1} and Gr\"atzer and Lakser~\cite{gGhL17}.

\begin{definition} Let $D$ be a finite distributive lattice. A subset $Q\subseteq D$ or, to be more precise, the pair $\pair QD$  is \emph{principal congruence representable} if there exist a finite lattice $L$ and an isomorphism $\phi\colon\Con(L)\to D$ such that $Q=\phi(\Princ(L))$. We say that $D$
 is \emph{fully principal congruence representable} if all subsets $Q$ of $D$ with $J^+(D)\subseteq Q$ are principal congruence representable.
\end{definition}

We introduce a seemingly stronger property of $D$ as follows. The \emph{automorphism group} of a lattice $L$ will be denoted by $\Aut(L)$.

\begin{definition} A  finite distributive lattice $D$ is 
\begin{equation}
\parbox{6cm}{\emph{fully principal congruence  representable with arbitrary automorphism groups,}}
\label{eqtxtfpcrDG}
\end{equation} 
in short, \emph{\eqref{eqtxtfpcrDG}-representable}, 
if for each subset $Q$ of $D$ such that $J^+(D)\subseteq Q$ and for any finite group $G$ such that $|D|=1\Rightarrow |G|=1$, there exist a finite lattice $L$ and an isomorphism $\phi\colon\Con(L)\to D$ such that $Q=\phi(\Princ(L))$ and $\Aut(L)$ is isomorphic to $G$.
\end{definition}

For more about full \pcrely, see 
Gr\"atzer~\cite{ggwith1},  Gr\"atzer and Lakser~\cite{gGhL17},  Cz\'edli, Gr\"atzer, and Lakser~\cite{czggghl}, and Cz\'edli~\cite{czgchrep}. The present paper relies  on these papers, in particular, it depends heavily on Gr\"atzer~\cite{ggwith1}. For related results on  the representability of the ordered set $Q$ as $\Princ(L)$ (without taking care of $D$), see
Cz\'edli~\cite{czgonemap}, \cite{czgaleph0}, \cite{czgprincout}, \cite{czgmanymaps}, and \cite{czgcometic} and    Gr\"atzer~\cite{gG13}, \cite{gGasm2016}, \cite{gGprincII}, and \cite{gGprincIII}.

Our main goal is to prove the following theorem. Our second target 
is to generalize the main result of Gr\"atzer~\cite{ggwith1}; see Theorem~\ref{thmmyscnD} in the present paper. Note that Theorem~\ref{thmmyscnD} will be needed to achieve the main goal.

\begin{theorem}[Main Theorem]\label{thmmain} 
If $D$ is a finite distributive lattice, then the following three conditions are equivalent.
\begin{enumeratei}
\item\label{thmmaina} $D$ is fully principal congruence  representable.
\item\label{thmmainb} $D$ is \eqref{eqtxtfpcrDG}-representable.
\item\label{thmmainc}
 $D$  is planar and it has at most one join-reducible coatom.
\end{enumeratei}
\end{theorem}

Combining this theorem with that of Cz\'edli~\cite{czgchrep}, we obtain the following statement; the definition of full chain-representability is postponed to the next section.

\begin{corollary} A finite distributive lattice is fully principal congruence  representable if and only if it is fully chain-representable.
\end{corollary}

\subsection*{Outline} In Section~\ref{sectsecnd},  we recall the main result of  Gr\"atzer~\cite{ggwith1} as Theorem~\ref{thmggBGrs} here, and we state its generalization in Theorem~\ref{thmmyscnD}. 
In a ``proof-by-picture" way, Section~\ref{sectprbpct} explains the construction required by 
the \eqref{thmmainc} $\Rightarrow$ \eqref{thmmaina} part of 
(the Main) Theorem~\ref{thmmain}. Even if  Section~\ref{sectprbpct} contains no rigorous proofs, it can rapidly convince the reader that our construction is ``likely to work''. In Section~\ref{sectqcol}, we recall the quasi-coloring technique from Cz\'edli~\cite{czgreprect} and develop it further.
In Section~\ref{sectdrve}, armed with quasi-colorings, we derive the \eqref{thmmaina} $\iff$ \eqref{thmmainc} part of  (the Main) Theorem~\ref{thmmain}. Section~\ref{secnewproof}  contains 
the proof of Theorem~\ref{thmmyscnD}; note that our method yields an alternative proof for Theorem~\ref{thmggBGrs}, taken from  Gr\"atzer~\cite{ggwith1}.  Finally, Section~\ref{sectlast} completes the proof of Theorem~\ref{thmmain}.

\section{G.\ Gr\"atzer's theorem and our second goal}\label{sectsecnd}
A subset $Q$ of a finite distributive lattice $D$ is a \emph{candidate subset} if $J^+(D)\subseteq Q$. 
By a \emph{$J(D)$-labeled chain}\footnote{Gr\"atzer~\cite{ggwith1} uses the terminology ``$J(D)$-colored'' but here by a ``coloring'' we shall mean a particular \emph{quasi-coloring}, which goes back to Cz\'edli~\cite{czgreprect}. According to our terminology, the map in \eqref{eqpbxNmchTrlab} is not a coloring in general.} we mean a triplet $\tripl C\labc D$ such that $C$ is a finite chain, $D$ is a finite distributive lattice, and 
\begin{equation}
\parbox{9cm}{$\labc\colon \Prime(C)\to J(D)$ is a surjective map from the \emph{set $\Prime(C)$ of all prime intervals} of $C$ onto $J(D)$.}
\label{eqpbxNmchTrlab}
\end{equation}
If $\inp\in\Prime(C)$, then $\labc(\inp)$ is the \emph{label} of the edge $\inp$. 
Given a $J(D)$-labeled chain $\tripl C\labc D$, we define a map denoted by $\erep$ from the \emph{set $\Intv(C)$ of all intervals} of  $C$ onto $D$ as follows: 
for $I\in\Intv(C)$,  let 
\begin{equation}
\erep(I):=\bigvee_{\inp\in\Prime(I)} \labc(\inp)\,;
\label{eqereP}
\end{equation}
the join is taken in $D$ and $\erep(I)$ is called the \emph{element represented} by $I$.
The set 
\begin{equation}
\srep(C,\labc,D):=\set{\erep(I): I\in \Intv(C)}
\label{eqSrEp}
\end{equation}
will be called the  \emph{set represented} by the $J(D)$-labeled chain $\tripl C\labc D$. Clearly, $\srep(C,\labc,D)$ is a candidate subset of $D$ in this case.  
A candidate subset $Q$ of $D$ is said to be \emph{chain-representable} if there exists a $J(D)$-labeled chain $\tripl C\labc D$ such that $Q=\srep(C,\labc, D)$. Note that $C$ need not be a subchain of $D$.

If $1_D\in J(D)$ and  $\tripl C\labc D$ is  \emph{$J(D)$-labeled chain}, then we define
a larger \emph{$J(D)$-labeled chain}   $\tripl {\cast}{\labc^\ast} D$ as follows: 
\begin{equation}
\parbox{10cm}{
we add a new largest element $1_{\cast}$ to $C$ to obtain 
$\cast=C\cup\set{1_{\cast}}$ and we extend $\labc$ to $\labc^\ast$ such that $\labc^\ast([1_C, 1_{\cast}])=1_D$.}
 \label{eqpbxnNghTrZ}
\end{equation}  
For elements $x,y$ and a prime interval $\inp$ of a lattice $L$, $\con_L(x,y)$ and $\con_L(\inp)$ denote the \emph{congruence generated} by $\pair xy$ and $\pair{0_\inp}{1_\inp}$, respectively. The subscript  is often dropped and we write $\con(x,y)$ and $\con(\inp)$.  A lattice $L$ will be called \emph{$\set{0,1}$-separating} if  for every $x\in L\setminus\set{0,1}$, $\con(0,x)=\con(x,1)$ is $1_{\Con(L)}$, the largest congruence of $L$.
The following result is due to Gr\"atzer~\cite{ggwith1}; note that its part \eqref{thmggBGrsc} is implicit in \cite{ggwith1}, but the reader
can find it by analyzing the construction  given in  \cite{ggwith1}. For a different approach, see the proof of Theorem~\ref{thmmyscnD} here.

\begin{theorem}[Gr\"atzer~\cite{ggwith1}]\label{thmggBGrs} Let $D$ be a finite distributive lattice. If $J^+\subseteq Q\subseteq D$, then the following two statements hold.
\begin{enumeratei}
\item\label{thmggBGrsa} If $Q\subseteq D$ is principal congruence representable, then it is chain-representable.
\item\label{thmggBGrsb} If $1=1_D$ is join-irreducible and $Q\subseteq D$ is chain-representable, then $Q\subseteq D$ is principal congruence representable.
\end{enumeratei} 
Furthermore, if  $Q\subseteq D$ is chain-representable and $1_D\in J(D)$, then
\begin{enumeratei}\setcounter{enumi}{2}
\item\label{thmggBGrsc} for every  $J(D)$-labeled chain $\tripl C\labc D$ representing $Q\subseteq D$, there exist a finite $\set{0,1}$-separating  lattice $L$ and an isomorphism $\phi\colon\Con(L)\to D$ such that 
\begin{enumerate}[\upshape(a)]
\item\label{thmggBGrsca}  $\phi(\Princ(L))=  \srep(C,\labc,D)= Q$, 
\item\label{thmggBGrscb} $\cast$ is a filter of $L$, \item\label{thmggBGrscc} $\labc^\ast(\inp)=\phi(\con_L(\inp))$ holds for every $\inp \in \Prime(\cast)$, and
\item\label{thmggBGrscd} for all $x\in \cast$ and $y\in L\setminus \cast$, if $y\prec x$, then $\con_L(y,x)=1_{\Con(L)}$.
\end{enumerate}
\end{enumeratei} 
\end{theorem}

\semmi{
We are going to give a new proof of Theorem~\ref{thmggBGrs}; this is motivated by the following three facts. First,  \ref{thmggBGrs}\eqref{thmggBGrsc}  is 
not stated in  Gr\"atzer~\cite{ggwith1} explicitly, so the reader has to analyze  the proof to find it. Second, in many cases, our method yields a smaller $L$ than the one constructed in \cite{ggwith1}. Third, we prove a little bit more; in fact, we are going to prove the following statement, which clearly implies Theorem~\ref{thmggBGrs} since the case $|D|=1$ is trivial.}

For the $1_D\in J(D)$ case, we are going to generalize  Theorem~\ref{thmggBGrs} as follows; note that if we did not care with $\Aut(L)$, then our lattice $L$ would often be smaller than the corresponding lattice constructed in  Gr\"atzer~\cite{ggwith1}.

\begin{theorem}\label{thmmyscnD} 
Let $D$ be a finite distributive lattice such that $1=1_D$ is join-irreducible and $|D|>1$, let $G$ be a finite group, and let $Q$ be candidate subset of $D$. If $Q\subseteq D$ is 
chain-representable, then for every 
 $J(D)$-labeled chain $\tripl C\labc D$
that represents $Q$, 
 there exist a finite  $\set{0,1}$-separating  lattice $L$ and an isomorphism $\phi\colon\Con(L)\to D$ such that 
\begin{enumeratei}
\item\label{thmmyscnDca}  $\srep(C,\labc,D)=\phi(\Princ(L))=Q$, 
\item\label{thmmyscnDcb} $\cast$,  which is defined in \eqref{eqpbxnNghTrZ}, is a filter of $L$, \item\label{thmggBGrscc} $\labc^\ast(\inp)=\phi(\con_L(\inp))$ holds for every $\inp \in \Prime(\cast)$, 
\item\label{thmmyscnDcd} for all $x\in \cast$ and $y\in L\setminus \cast$, if $y\prec x$, then $\con_L(y,x)=1_{\Con(L)}$, and
\item\label{thmmyscnDce} $\Aut(L)$ is isomorphic to $G$.
\end{enumeratei} 
\end{theorem}

\begin{figure}[ht] 
\centerline
{\includegraphics[scale=1.0]{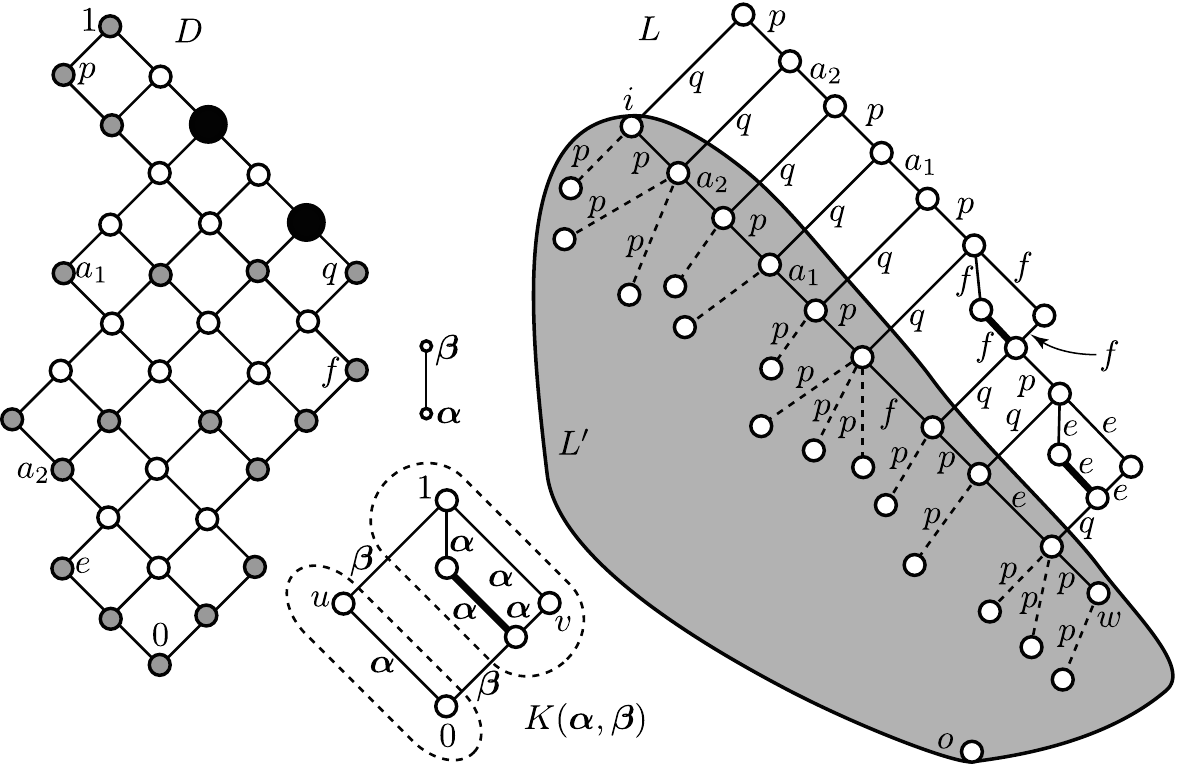}}
\caption{{$D$ with two coatoms,  $K(\balpha,\bbeta)$, and $L$ that we construct}}
\label{figexone}
\end{figure}

\section{From $1\in J(D)$ to $1\notin J(D)$, a proof-by-picture approach}\label{sectprbpct}
In this section, we outline our construction that derives the \eqref{thmmainc} $\Rightarrow$  \eqref{thmmaina} part of Theorem~\ref{thmmain} from Theorem~\ref{thmggBGrs}. 
Since the $1\in J(D)$ case follows from the conjunction of Cz\'edli~\cite{czgchrep} and  Gr\"atzer~\cite{ggwith1}, here we deal only with the case where $1=1_D$ is join-reducible.  So, in this section, we assume that $D$ is a planar distributive lattice such that $1=1_D$ is join-reducible. It belongs to the folklore that every element $x$ of $D$ covers at most two elements and $x$ is the join of at most two join-irreducible elements. Hence, there are distinct $p,q\in J(D)$ such that $1_D=p\vee q$ and $p\prec 1$; see Figure~\ref{figexone}. Also, let $Q\subseteq D$ such that $J^+(D)\subseteq Q$.
In the figure, $Q$ consists of the grey-filled and the  large black-filled elements. Let us denote by $D'$ the principal ideal $\ideal p=\set{d\in D: d\leq p}$, and let $Q'=Q\cap D'$.
It will not be hard to show that 
\begin{equation}
\parbox{9.6cm}{the filter
$\filter q=\set{d\in D: d\geq q}$ is a chain,  
$D$ is the disjoint union of $D'$ and $\filter q$, and $q$ is a maximal element of $J(D)$,} 
\label{eqpbxflRchN}
\end{equation}
as shown in the figure. 
Next, we focus on $(Q\cap \filter q)\setminus\set q$; it consists of the large black-filled elements in the figure. By the maximality of $q$ in $J(D)$, these  elements are join-reducible, whereby each of them is the join of  $q$ and another join-irreducible element $a_i$. 
In our case, 
$(Q\cap  \filter q)\setminus\set q=\set{a_1\vee q, a_2\vee q}$; in general, it is $\set{a_1,\dots, a_k}$ where $k\geq 0$.  We will show that
\begin{equation}
\text{$J(D')\cap\ideal q$ has at most two maximal elements.}
\label{eqtxtTwCWmnt}
\end{equation}
Let $\set{e,f}$ be the set of maximal elements of $J(D')\cap\ideal q$; note that $e=f$ is possible but causes no problem.

We know from Cz\'edli~\cite{czgchrep} that $Q'\subseteq D'$ is represented by a $J(D')$-labeled chain $\tripl{C_0}{\labc_0'}{D'}$. 
Let $C_1$ be the chain of length $2k+4=8$ whose edges, starting from below, are colored by $p,e,p,f,p,a_1,p,a_2$. The 
\begin{equation}
\parbox{8.2cm}{\emph{glued sum} $C:=C_0\mathop{\dot+}C_1$ is obtained from their summands by putting $C_1$ atop $C_0$ and identifying the top element of $C_0$ with the bottom element of $C_1$.}
\label{eqtxtGluedSm}
\end{equation}
In this way, we have obtained a $J(D')$-labeled chain $\tripl C{\labc'}{D'}$.
It will be easy to show that 
\begin{equation}
\text{$\tripl C{\labc'}{D'}$ also represents $Q'\subseteq D'$.}
\label{eqtxtCsprsTsskZ}
\end{equation}
Thus, Theorem~\ref{thmggBGrs} yields a finite lattice $L'$ and an isomorphism $\phi'\colon \Con(L')\to D'$ such that \ref{thmggBGrs}\eqref{thmggBGrsc} holds with $\kvint{Q'}{D'}{C}{L'}{\phi'}$ instead of $\kvint{Q}{D}{C}{L}{\phi}$. In the figure, $L'$ is represented by the grey-filled area on the right. In $C$, there is a unique element $w$ such that $C_0=\ideal w$ and $C_1=\filter w$ (understood in $C$, not in $L'$); only $\filter w$, which is a filter of $L'$ and also a filter of $\cast$ is indicated in the figure. Since $p=1_{D'}$, the top edge of $\filtrof{D'} w$ (the filter understood in $D'$) is $p$-labeled. Some 
elements outside $\cast$ that are covered by elements of $\filter w$ are also indicated in the figure; the covering relation in these cases are shown by dashed lines; \ref{thmggBGrs}\eqref{thmggBGrscd} and $p=1_{D'}$ motivate that these edges are labeled by $p$.

Next, by adding $3k+7=13$ new elements to $L'$, we obtain a larger lattice $L$, as indicated in Figure~\ref{figexone}. Each of the edges labeled by $q$ generate the same congruence, which we denote by $\what q$. Consider the lattice $K(\balpha,\bbeta)$ in the middle of  Figure~\ref{figexone}. For later reference, note that the only property of this lattice that we will use is that 
\begin{equation}
\text{$K(\balpha,\bbeta)$ has exactly one nontrivial congruence,}
\end{equation}
$\balpha$, whose blocks are indicated by dashed ovals.
Hence, if this lattice is a sublattice of $L$, then any of its $\balpha$-colored edge generates a congruence that is smaller than or equal to the congruence generated by a $\bbeta$-colored edge. Copies of this lattice ensure the following two ``comparabilities''
\begin{equation}
\text{$\phi'^{-1}(e)\leq \what q$ and $\phi'^{-1}(f)\leq \what q$.}
\label{eqtxtdzBgRsW}
\end{equation}
For $x\parallel y\in L$, if $x$ and $y$ cover their meet and are covered by their join, then $\set{x\wedge y,x,y,x\vee y}$ is a \emph{covering square} of $L$. For $i\in{1,\dots,k}=\set{1,2}$, 
\begin{equation}
\parbox{9cm}{the covering squares with $a_i,q,a_i,q$-labeled edges guarantee that $\phi'^{-1}(a_i)\vee\what q$ is a principal congruence.}
\label{eqtxthhRsK}
\end{equation} 
It will be easy to see that our construction yields all comparabilities and principal congruences that we need. We will rigorously prove that we do not get more comparabilities and principal congruences than those described in \eqref{eqtxtdzBgRsW} and \eqref{eqtxthhRsK}. Thus, it will be straightforward to conclude the  \eqref{thmmainc} $\Rightarrow$  \eqref{thmmaina} part of Theorem~\ref{thmmain}

\section{Quasi-colored lattices}\label{sectqcol}
Reflexive and transitive  relations are called \emph{quasiorderings}, also known as 
 \emph{preorderings}. If $\nu$ is a quasiordering on a set $A$, then $\ppair A\nu$ is said to be a \emph{quasiordered set}. For $H\subseteq A^2$, the least quasiordering of $A$ that includes $H$ will be denoted by  $\preogen_A(H)$, or simply by  $\preogen(H)$ if there is no danger of confusion. For $H=\set{\pair ab}$, we will of course write $\preogen(a,b)$.  
Quite often, especially if we intend to exploit the transitivity of $\nu$,  we  write $a\leq_\nu b$ or $b\geq_\nu a$  instead of $\pair a b\in\nu$. 
Also, $a=_\nu b$ will stand for  $\set{\pair ab, \pair ba}\subseteq \nu$.
The set of all quasiorderings on $A$ form a complete lattice $\Quo(A)$ under set inclusion. 
For $\nu,\tau\in\Quo(A)$, the join $\nu\vee \tau$ is $\preogen(\nu\cup\tau)$. \emph{Orderings} are antisymmetric quasiorderings, and a set with an ordering is an \emph{ordered set}, also known as a \emph{poset}.
Following Cz\'edli~\cite{czgreprect}, 
a \emph{quasi-colored lattice} is a lattice $L$ of finite length 
together with a  surjective map $\gamma$, called a \emph{quasi-coloring}, from $\Prime(L)$ onto a 
quasiordered set $\ppair H\nu)$ such that
for all $\inp,\inq\in\Prime(L)$, 
\begin{enumerate}[xxxxx]
\item[\colgammacon] if $\gamma(\inp)\geq_\nu \gamma(\inq)$, then  $\con(\inp)\geq\con(\inq)$, and
\item[\colcongamma] if $\con(\inp)\geq\con(\inq)$, then $\gamma(\inp)\geq_\nu\gamma(\inq)$.
\end{enumerate}
The values of $\gamma$ are called \emph{colors} (rather than quasi-colors). 
If $\gamma(\inp)=b$, then we say that $\inp$ is colored by $b$.
In figures, the colors of (some) edges are indicated by labels. 
Note the difference: even if the colors are often given by labels, a labeling like \eqref{eqpbxNmchTrlab} need not be a quasi-coloring. 
If $\ppair H\nu$ happens to be an ordered set, then $\gamma$ above is a \emph{coloring}, not just a quasi-coloring. The map $\natc$ from $\Prime(L)$ to $J(\Con(L))=\ppair{J(\Con(L))}{\leq}$,
defined by $\natc(\inp):=\con(\inp)$, is the so-called \emph{natural coloring} of $L$.

The relevance of quasi-colorings of a lattice $L$ of finite length lies in the fact that they determine $\Con(L)$; see  Cz\'edli~\cite[(2.8)]{czgreprect}. Even if we will use quasi-colorings in our stepwise constructing method, we need only the following statement.

\begin{lemma}\label{lemmawHGhwhg}
If  $L$ and $D$ are finite lattices, $D$ is distributive, and   $\whg\colon \Prime(L)\to J(D)$ is a coloring, then the map
$\mu\colon \ppair{J(\Con(L))}{\leq}\to \ppair{J(D)}{\leq}$,
defined by $\con(\inp)\mapsto \whg(\inp)$ where $\inp\in\Prime(L)$, is an order isomorphism.
\end{lemma}

\begin{proof} It is well known that $J(\Con(L))=\{\con(\inp): \inp\in\Prime(L)\}$. We obtain from \colcongamma{} that $\mu$ is well defined, that is, 
if $\con(\inp)=\con(\inq)$, then $\whg(\inp)=\whg(\inq)$. Furthermore, \colcongamma{} gives that $\mu$ is order-preserving. It is surjective since so is $\whg$. We conclude from \colgammacon{} that $\mu(\con(\inp))\leq \mu(\con(\inq))$ implies that $\con(\inp)\leq\con(\inq)$. This also yields that $\mu$ is injective.
\end{proof}

Next,  assume that $\ppair{A_1}{\nu_1}$ and  $\ppair{A_2}{\nu_2}$ are quasiordered sets. By a \emph{homomorphism} $\delta\colon  \ppair{A_1}{\nu_1}\to \ppair{A_2}{\nu_2}$ we mean  a map $\delta\colon A_1\to A_2$ such that $\delta(\nu_1)\subseteq\nu_2$, that is, $\pair{\delta(x)}{\delta(y)}\in\nu_2$ holds for all $\pair xy\in\nu_1$. Following G.\ Cz\'edli  and A.~Len\-ke\-hegyi~\cite{r:czg-lenk}, 
\begin{equation}\label{e:diK245yatwd2kH}
\dker(\delta):=\bigl\{ \pair x y\in A_1^2: \bigl(g(x),g(y)\bigr)\in\nu_2\bigr\}
\end{equation}
is called the \emph{directed kernel} of $\delta$. Clearly, it is a quasiordering on $A_1$. 
Note that $\delta$ is a homomorphism if and only if 
 $\dker (\delta)\supseteq \nu_1$. 
The following lemma, which we need later, is Lemma 2.1 in Cz\'edli~\cite{czgreprect}. Note that we compose maps from right to left.

\begin{lemma}[\cite{czgreprect}]\label{lMa:wi6p}
Let $M$ be a finite lattice, and let  $\ppair Q\nu$ and $\ppair P\sigma$ be quasiordered sets. Let $\gamma_0\colon\Prime(M)\to \ppair Q\nu$ be a quasi-coloring. Assume that  $\delta\colon\ppair Q\nu\to \ppair P\sigma$ is a surjective homomorphism such that $\dker (\delta)\subseteq \nu$. Then the composite map $\delta\circ\gamma_0\colon \Prime(M)\to \ppair P\sigma$ is a quasi-coloring.
\end{lemma}

The advantage of quasi-colorings over colorings is that, as opposed to orderings, quasiorderings form a lattice; see Cz\'edli~\cite[page 315]{czgreprect} 
for more motivation. Another motivating fact is given by the following lemma. If $L_1$ is an ideal and $L_2$ is a filter of a lattice $L$ such that $L_1\cup L_2=L$ and $L_1\cap L_2\neq \emptyset$, then $L$ is the (\emph{Hall--Dilworth}) \emph{gluing} of $L_1$ and $L_2$ over their intersection.

\begin{lemma}\label{lemmaHDqcol}
Let $L$ be a lattice of finite length such that it is the Hall--Dilworth gluing of $L_1$ and $L_2$  over $L_1\cap L_2$. For $i\in\set{1,2}$, let $\gamma_i\colon\Prime(L_i)\to\ppair{H_i}{\nu_i}$ be a quasi-coloring, and assume that 
\begin{equation}
H_1\cap H_2\subseteq \set{\gamma_1(\inp):\inp \in \Prime(L_1\cap L_2)\text{ and } \gamma_1(\inp)=\gamma_2(\inp)}.
\label{eqnTrsVbmzQ}
\end{equation}
Let $H:=H_1\cup H_2$, and define $\gamma\colon \Prime(L)\to H$ by the rule 
\begin{equation}
\gamma(\inp)=
\begin{cases}
\gamma_1(\inp) &\text{ for }\inp\in \Prime(L_1),\cr
\gamma_2(\inp) &\text{for }\inp\in \Prime(L)\setminus\Prime(L_1). 
\end{cases}
\label{eqddvcmCsghH}
\end{equation}
Let
\begin{equation}
\nu=\preogen\bigl(\nu_1\cup \nu_2\cup\set{\pair{\gamma_j(\inp)}{\gamma_{3-j}(\inp)}: \inp\in\Prime(L_1\cap L_2),\,\,\,j\in\set{1,2} }\bigr).
\label{eqnghmnNz}
\end{equation}
Then $\gamma\colon\Prime(L)\to \ppair H\nu$ is a quasi-coloring. 
\end{lemma}

Note the following three facts. In \eqref{eqddvcmCsghH}, the subscripts 1 and 2 could be interchanged and even a ``mixed'' definition of $\gamma(\inp)$ would work. Even if $\gamma_1$ and $\gamma_2$ are colorings, $\gamma$ in Lemma~\ref{lemmaHDqcol} is only a quasi-coloring in general. 
The case where $|L_1|=|L_2|=2=|L|-1$ and $H_1=H_2$ exemplifies that the assumption \eqref{eqnTrsVbmzQ} cannot be omitted.

%
%
%
%

Before proving this lemma, we recall a useful statement from Gr\"atzer~\cite{ggprimeprojective}.
For $i\in \set{1,2}$, let $\inp_i=[x_i,y_i]$
be prime intervals of a lattice $L$. We say that $\inp_1$ is \emph{prime-perspective down} to $\inp_2$, denoted by $\inp_1 \ppdn \inp_2$ or $\pair{x_1}{y_1}\ppdn \pair{x_2}{y_2}$,  if 
$y_1=x_1\vee y_2$ and $x_1\wedge y_2\leq x_2$; see Figure~\ref{fig-pp}, where the solid lines indicate prime intervals while the dotted ones stand for the ordering relation of $L$. We define \emph{prime-perspective up}, denoted by  $\inp_1 \ppup \inp_2$, dually. The reflexive transitive closure of the union of $\ppup$ and $\ppdn$ is called \emph{prime-projectivity}.
%

\begin{figure}[ht]
\centerline
{\includegraphics[scale=1.0]{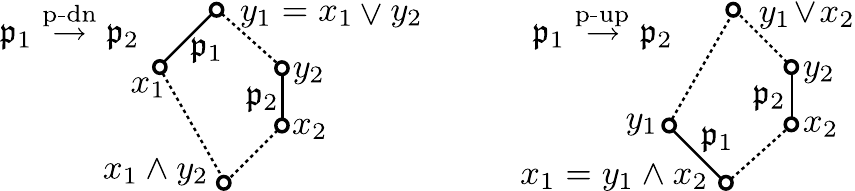}}
\caption{Prime perspectivities   \label{fig-pp}}
\end{figure}

\begin{lemma}[{Prime-Projectivity Lemma; see  Gr\"atzer~\cite{ggprimeprojective}}]
\label{primprojlemma}
Let $L$ be a lattice of finite length, and let $\inr_1$ and $\inr_2$ be prime intervals in $L$. 
Then $\con(\inr_1)\geq \con(\inr_2)$ 
if and only if there exist an $n\in\mathbb N_0$ 
and a sequence  $\inr_1=\inp_0$, $\inp_1$, \dots, $\inp_n=\inr_2$ of prime intervals such that for each $i\in\set{1,\dots, n}$, $\inp_{i-1}\ppdn \inp_i$ or $\inp_{i-1}\ppup \inp_i$.
\end{lemma}

\begin{proof}[Proof of Lemma~\ref{lemmaHDqcol}] In order to prove
\colgammacon{}, let $\inp$ and $\inq$ be prime intervals of $L$ such that $\gamma(\inp)\geq_\nu\gamma(\inq)$. By the definition of $\nu$, 
there is a sequence $\gamma(\inp)=h_0, h_1, h_2,\dots, h_k=\gamma(\inq)$ in $H$ such that, for each $i$, ${h_{i-1}}\geq_{\nu_1}{h_i}$,  or ${h_{i-1}}\geq_{\nu_2}{h_i}$, or $\pair{h_{i-1}}{h_i}=\pair{\gamma_{j}(\inr_{i-1})}{\gamma_{3-j}(\inr_i')}$  for some $j=j(i)\in\set{1,2}$ and  $\inr_{i-1}=\inr_i'\in\Prime(L_1\cap L_2)=\Prime(L_1)\cap \Prime(L_2)$. 
In the first case, by the surjectivity of $\gamma_1$ and the satisfaction of \colgammacon{} in $L_1$, we can pick  prime intervals $\inr_{i-1},\inr_i'\in\Prime(L_1)$ such that  $\gamma_1(\inr_{i-1})=h_{i-1}$, $\gamma_1(\inr'_{i})=h_{i}$, and   $\con_{L_1}(\inr_{i-1})\geq \con_{L_1}(\inr_i')$.  
In the second case, we obtain similarly that $\gamma_2(\inr_{i-1})=h_{i-1}$, $\gamma_2(\inr'_{i})=h_{i}$, and  $\con_{L_2}(\inr_{i-1})\geq \con_{L_2}(\inr_i')$ for some $\inr_i,\inr_i'\in\Prime(L_2)$. In the third case, both $\con_{L_1}(\inr_{i-1})\geq \con_{L_1}(\inr_i')$ and $\con_{L_2}(\inr_{i-1})\geq \con_{L_2}(\inr_i')$ trivially hold, since 
$\inr_{i-1}=\inr_i'$.

Hence, for  for every $i$ in $\set{1,\dots,2}$, 
Lemma~\ref{primprojlemma} gives us 
\begin{equation}
\text{a ``prime-projectivity sequence'' from $\inr_{i-1}$ to $\inr_i'$.}
\label{eqtxtppsQng}
\end{equation}
Since $\Prime(L_1)\subseteq \Prime(L)$ and $\Prime(L_2)\subseteq \Prime(L)$, this sequence is in $\Prime(L)$. We claim that, for each $i\in\set{1,\dots,k}$, 
\begin{equation}
\text{there is a prime-projectivity sequence from $\inr_{i}'$ to $\inr_i$.}
\label{eqtxtppsjzTn}
\end{equation}
In order to verify this, assume first that 
$\gamma_1(\inr_i')=h_i=\gamma_1(\inr_i)$. Then 
 ${\gamma_1(\inr_i')}  \geq_{\nu_1}{\gamma_1(\inr_i)}$ and the validity of \colgammacon{} for $\gamma_1$ imply that $\con_{L_1}(\inr_i')\geq\con_{L_1}(\inr_i)$, whereby  \eqref{eqtxtppsjzTn} follows from  Lemma~\ref{primprojlemma}. The case $\gamma_2(\inr_i')=h_i=\gamma_2(\inr_i)$ is similar. Hence, we can assume that $\gamma_j(\inr_i')=h_i=\gamma_{3-j}(\inr_i)$ for some $j\in\set{1,2}$. 
Clearly,  $h_i$ is in $H_1\cap H_2$, since it is in the range of $\gamma_j$ and that of  $\gamma_{3-j}$. By \eqref{eqnTrsVbmzQ}, we can pick a prime interval $\inr_i''\in\Prime(L_1\cap L_2)$ such that 
$\gamma_j(\inr_i'')=h_i=\gamma_{3-j}(\inr_i'')$. 
Applying \colgammacon{} and Lemma~\ref{primprojlemma} to $\gamma_j$, we obtain that 
there is a  prime-projectivity sequence from $\inr_{i}'$ to $\inr_i''$. Similarly,  we obtain that there is another sequence from $\inr_{i}''$ to $\inr_i$. Concatenating these two sequences, we obtain a prime-projectivity sequence from $\inr_{i}'$ to $\inr_i$. 
This shows the validity of \eqref{eqtxtppsjzTn}. Finally, concatenating the sequences from \eqref{eqtxtppsQng} and those from 
\eqref{eqtxtppsjzTn}, we obtain a prime-projectivity sequence from $\inp=\inr_0$ to $\inq=\inr_k$. So the easy direction of  Lemma~\ref{primprojlemma} implies that $\con(\inp)\geq \con(\inq)$, proving that $L$ satisfies \colgammacon{}.

Observe that 
\begin{equation}
\parbox{8cm}{
for all $i\in \set{1,2}$ and $\inr\in\Prime(L_i)$, 
$\,\gamma_i(\inr)=_\nu \gamma(\inr)$;}
\label{eqpbxdPmzTs}
\end{equation}
this is clear either because $\inr\notin\Prime(L_{3-i})$, or because $\inr\in\Prime(L_1\cap L_2)$ and 
 $\gamma_1(\inr)=\gamma_2(\inr)=\gamma(\inr)$ or 
$\pair{\gamma_{i}(\inr)}{\gamma_{3-i}(\inr)}\in\nu$ by \eqref{eqnghmnNz}.

Next, in order to prove that $L$ satisfies 
\colcongamma{}, assume that $\inp,\inq\in\Prime(L)$ such that $\con(\inp)\geq \con(\inq)$. We need to show that $\gamma(\inp)\geq_\nu\gamma(\inq)$. 
This is clear if $\inp=\inq$. 
Since $\nu$ is transitive, Lemma~\ref{primprojlemma} and duality allow us to assume that $\inp\ppup \inq$.  We are going  to deal only with the case $\inp\in\Prime(L_1)\setminus\Prime(L_2)$ and  $\inq\in\Prime(L_2)\setminus\Prime(L_1)$, since the cases  $\set{\inp,\inq}\subseteq\Prime(L_1)$ and $\set{\inp,\inq}\subseteq\Prime(L_2)$ are much easier 
while the case $\inp\in\Prime(L_2)\setminus\Prime(L_1)$ and  $\inq\in\Prime(L_1)\setminus\Prime(L_2)$ is excluded by the upward orientation of the prime-perspectivity. So $\inp=[x_1,y_1]=[y_1\wedge x_2,y_1]$ and
$\inq =[x_2,y_2]$ with $y_2\leq y_1\vee x_2$; see Figure~\ref{fig-pp}. Clearly, $x_1,y_1\in L_1\setminus L_2$ and $x_2,y_2\in L_2\setminus L_1$. By the description of the ordering relation in Hall--Dilworth gluings, we can pick an $x_3\in L_1\cap L_2$ such that $x_1\leq x_3\leq x_2$. Let $y_3:=y_1\vee x_3$. It is in $L_1\cap L_2$ since $L_1$ is a sublattice and $L_2$ is a filter in $L$. 
Since $x_3\vee x_2 =x_2\leq y_2\leq y_1\vee x_2= y_1\vee  (x_3\vee x_2)=(y_1\vee x_3)\vee x_2= y_3\vee x_2$, we have that $\con(\inq)=\con({x_2},{y_2})\leq \con({x_3},{y_3})$. Combining this inequality,  the well-known rule that 
\begin{equation}
\parbox{8cm}{in every finite distributive lattice $D$,\\ 
$(a\in J(D)$  and $a\leq b_1\jj\dots\jj b_n)\Longrightarrow (\exists i)(a\leq b_i)$,}
\label{eqdjlfkWsDbF}
\end{equation}
$\con(x_3,y_3)=\bigvee\set{\con(\inr): \inr\in\Prime([x_3,y_3])}$, and the distributivity of $\Con(L)$,  
we obtain a prime interval $\inr\in \Prime([x_3,y_3])\subseteq \Prime(L_1\cap L_2)$ such that
$\con(\inq)\leq \con(\inr)$. Since $\con(\inp)=\con(x_1,y_1)$ collapses  $(x_3,y_3)=(x_1\vee x_3, y_1\vee x_3)$, it collapses $\inr$. Hence, $\con(\inp)\geq \con(\inr)$. Since $\gamma_1$ is a quasi-coloring, this inequality and \eqref{eqpbxdPmzTs} yield that $\gamma(\inp)=\gamma_1(\inp)\geq_{\nu_1} \gamma_1(\inr) =_\nu \gamma(\inr)$. Hence, $\gamma(\inp)\geq_\nu  \gamma(\inr)$. Since $\gamma_2$ is also a quasi-coloring, the already established $\con(\inr)\geq \con(\inq)$ leads to $\gamma(\inr)\geq_\nu  \gamma(\inq)$ similarly. Thus, by transitivity, $\gamma(\inp)\geq_\nu  \gamma(\inq)$, showing that $L$ satisfies \colcongamma{}.
\end{proof}

\section{Deriving the Main Theorem from Theorem~\ref{thmggBGrs}}\label{sectdrve}
Based on  Theorem~\ref{thmggBGrs} and the plan outlined in Section~\ref{sectprbpct}, 
this section is devoted to the following proof.

\begin{proof}[First part of the proof of Theorem~\ref{thmmain}]
The implication \eqref{thmmainb} $\Rightarrow$ \eqref{thmmaina} is trivial, and it was proved in Cz\'edli~\cite{czgchrep} that \eqref{thmmaina} implies \eqref{thmmainc}. So we  need to show only that   \eqref{thmmainc} implies  \eqref{thmmainb}.
In this section we show only that  \eqref{thmmainc} implies  \eqref{thmmaina}; we will explain
in Section~\ref{sectlast}
 how to modify our argument to yield the required implication  \eqref{thmmainc} $\Rightarrow$  \eqref{thmmainb}.

In order to prove that  \eqref{thmmainc} $\Rightarrow$  \eqref{thmmaina}, let $D$ be an arbitrary planar distributive lattice with at most one join-reducible atom. We can assume that $|D|>1$ since the opposite case is trivial.

First, assume that $1_D\in J(D)$. We know from Cz\'edli~\cite{czgchrep} that for every $Q$, if 
$J^+(D)\subseteq Q\subseteq D$, then the inclusion $Q\subseteq D$ is chain-representable. Hence, by Theorem~\ref{thmggBGrs}, it is principal congruence representable, as required. 

%
%
%
%

Second, assume that $1_D\notin J(D)$. Let $Q$ be a subset of $D$ such that $J^+(D)\subseteq Q$.
In order to obtain a lattice $L$ that 
witnesses the principal congruence representability of $Q\subseteq D$, we do exactly the same as in Section~\ref{sectprbpct}; of course, now we cannot assume that $k=2$. For the sake of contradiction, suppose that $q$ is not a maximal element of $J(D)$, and pick a $q'\in J(D)$ such that $q<q'$. Then $q'\leq 1_D=p\vee q$ and \eqref{eqdjlfkWsDbF} imply that $q'\leq p$, whereby $q<q'\leq p$ yields that $1_D=p\vee q=p\in J(P)$, which is a contradiction. Thus, $q$ is a maximal element of $J(D)$. 
We know from the folklore or, say, from  Cz\'edli and Gr\"atzer~\cite{czgggchapter} that
\begin{equation}
\text{$J(D)$ is the union of two chains.}
\label{eqtxttwCnS}
\end{equation}
So we have two chains $C_1$ and $C_2$ such that $J_0(D)=C_1\cup C_2$ and $0\in C_1\cap C_2$. Let, say, $q\in C_2$. If $x\in \filter q$ and $x\neq q$, then $x=y_1\vee y_2$ for some $y_1\in C_1$ and 
$y_2\in C_2$. Since $q$ is a maximal element of $J(D)$, $y_2\leq q$ and $x=y_1\vee q$. For $x=q$, we can let $y_1=0\in C_1$. Hence, $\filter q\subseteq \set{z\vee q:  z\in C_1}$.  Since $C_1$ is a chain, 
so are $\set{z\vee q:  z\in C_1}$ and its subset $\filter q$. 
Finally, $D\cap\filter 0=\emptyset$ follows from $p\nleq q$. The facts established so far prove \eqref{eqpbxflRchN}.  

Observe that \eqref{eqtxttwCnS} implies  \eqref{eqtxtTwCWmnt}. Note that even if $e=f$, we will use the lattice given on the left of Figure~\ref{figrZh}. This will cause no problem since then $f$ can be treated as an alter ego of $e$, similarly to the alter egos $p_1,\dots, p_{k+3}$, see later, of $p$.

In order to verify \eqref{eqtxtCsprsTsskZ}, observe that $\Intv(C_0)
\subseteq \Intv(C)$ implies that the inclusion  $\srep(C_0,\labc_0',D')\subseteq \srep(C,\labc',D')$ holds; see \eqref{eqSrEp} for the notation. To see the converse inclusion, let $I\in \Intv(C)$. 
If $\length(I)\leq 1$, then $\erep(I)\in J_0(D')\subseteq \srep(C_0,\labc_0',D')$ is clear. If $\length(I)\geq 2$, then either  $I\in \Intv(C_0)$ and $\erep(I)\in \srep(C_0,\labc_0',D')$ is obvious, or $I\notin \Intv(C_0)$ and we have that 
$\erep(I)=p\in J_0(D')\subseteq \srep(C_0,\labc_0',D')$. Therefore, \eqref{eqtxtCsprsTsskZ} holds.

\begin{figure}[ht] 
\centerline
{\includegraphics[scale=1.0]{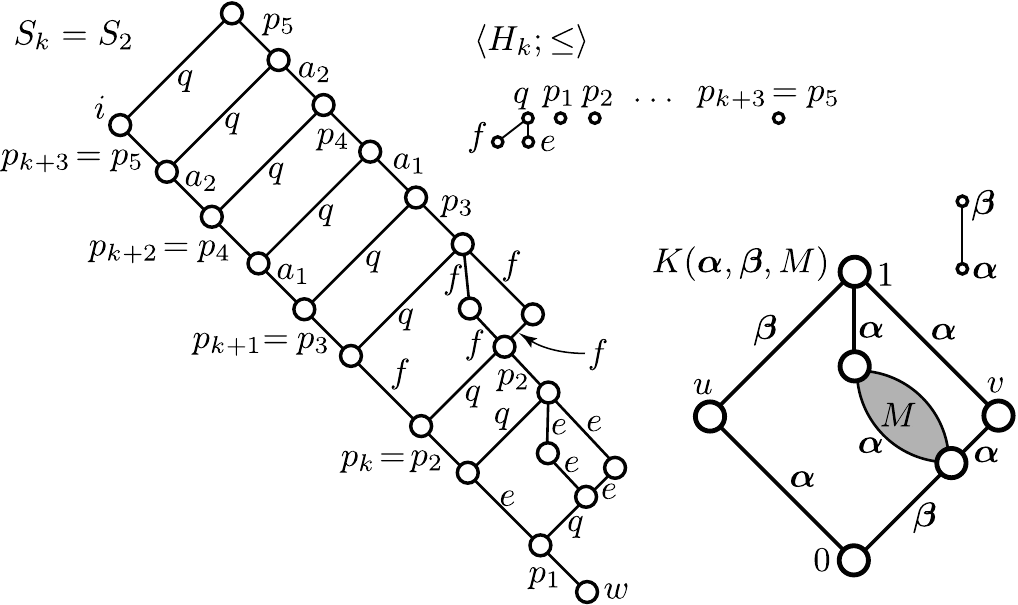}}
\caption{$S_k$ for $k=2$, $\ppair{H_k}{\leq}$, and $K(\balpha,\bbeta,M)$}
\label{figrZh}
\end{figure}

\begin{figure}[ht] 
\centerline
{\includegraphics[scale=1.0]{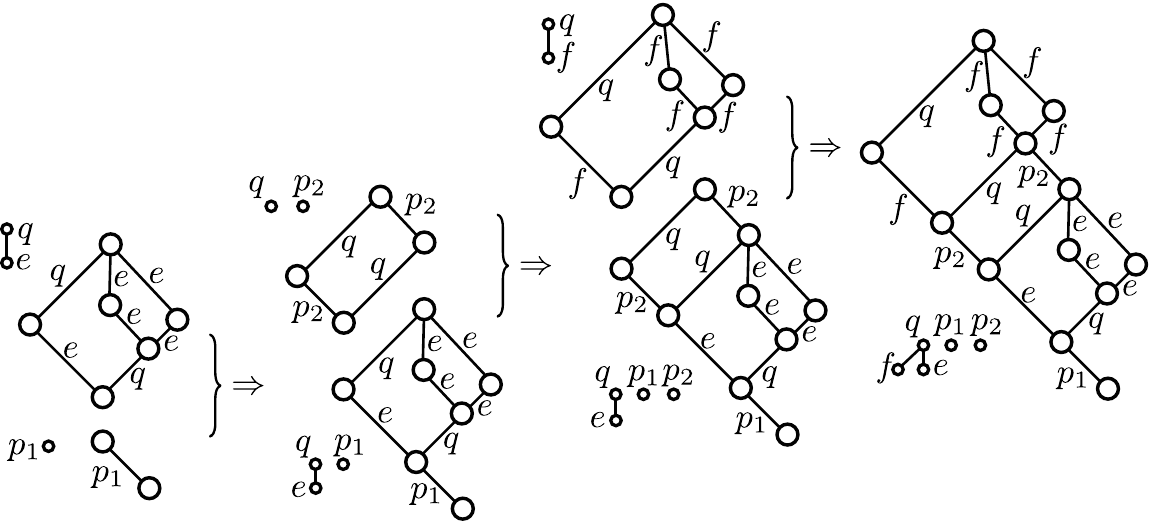}}
\caption{{Elementary steps towards  \eqref{eqtxtgKcLrNg}}}
\label{figgjWh}
\end{figure}

It is straightforward to check that $K(\balpha,\bbeta)$ is colored (not only quasi-colored) by the two-element chain $\set{\balpha<\bbeta}$, as indicated in Figure~\ref{figexone}.  Of course, we can rename the elements of this chain. 
For later reference, let $M$ be a simple lattice, and let $K(\balpha,\bbeta,M)$ denote the colored lattice we obtain from $K(\balpha,\bbeta)$ so that we replace its thick prime interval, see Figure~\ref{figexone}, by $M$ as indicated in Figure~\ref{figrZh}; all edges of $M$ are colored by $\balpha$. 
In Section~ \ref{sectlast}, we will rely on the obvious fact that 
\begin{equation}
\parbox{8cm}{whatever we do with $K(\balpha,\bbeta)$ in this section, we could do it with  $K(\balpha,\bbeta,M_1)$, $K(\balpha,\bbeta,M_2)$, \dots, where $M_1$, $M_2$,\dots{} are finite simple lattices.}
\label{eqpbxshwdZhTc}
\end{equation}

Let $S_k$ be the lattice given in Figure~\ref{figrZh}. Also, this figure defines an ordered set $\ppair{H_k}{\leq}$. Let $\who\gamma_k\colon \Prime(S_k)\to H_k$ be given by the labeling; we claim that 
\begin{equation}
\text{$\who\gamma_k\colon \Prime(S_k)\to \ppair{H_k}{\leq}$ is a coloring; see Figure~\ref{figrZh}.}
\label{eqtxtgKcLrNg}
\end{equation}
Note that the \emph{colors}, see \eqref{eqtxtgKcLrNg}, of the edges of the chain $[w,i]$ of $S_k$ are the same as the \emph{labels}, see \eqref{eqpbxnNghTrZ}, of the edges of the  corresponding filter of $\cast$.
We obtain \eqref{eqtxtgKcLrNg} by applying Lemma~\ref{lemmaHDqcol} repeatedly; the first three steps are given if Figure~\ref{figgjWh}; the rest of the steps are straightforward. In Figure~\ref{figgjWh}, going from left to right, we construct larger and larger quasi-colored (in fact, colored) lattices by Hall-Dilworth gluing. The colors are given by labeling and their ranges by small diagrams in which the elements are given by half-sized little circles. The action of gluing is indicated by ``$\} \Rightarrow$''. Now that we have decomposed the task into elementary steps, we conclude \eqref{eqtxtgKcLrNg}.

Next, let $\natc'\colon \Prime(L')\to J(\Con(L'))$ be the natural coloring of $L'$, that is, for $\inp\in \Prime(L')$, we have that $\natc'(\inp)=\con_{L'}(\inp)$. Since $\phi'\colon\Con(L')\to D'$ is a lattice isomorphism, see after \eqref{eqtxtCsprsTsskZ}, its restriction $\psi':=\restrict{\phi'}{J(\Con(L'))}$ is an order isomorphism from $J(\Con(L'))$ onto $J(D')$. Therefore, the composite map $\gamma_1:= \psi'\circ\natc'$  is a coloring $\gamma_1\colon \Prime(L')\to \ppair{J(D')}{\leq}$.

By \ref{thmggBGrs}\eqref{thmggBGrscb},  $\cast$ is a filter of $L'$, whereby the filter $\filtrof{L'}w$
is a chain. Hence, $L$ is the Hall--Dilworth gluing of $L'$ and $S_k$; compare Figures~\ref{figexone} and \ref{figrZh}. As a preparation to the next application of Lemma~\ref{lemmaHDqcol}, we denote the ordered sets $\ppair{J(D')}{\leq}$ and $\ppair{H_k}{\leq}$ also by $\ppair{J(D')}{\nu_1}$ and $\ppair{H_k}{\nu_2}$, respectively. We let $\gamma_2=\who\gamma_k$; see \eqref{eqtxtgKcLrNg}. 
Finally, $L'$ and $S_k$ will also be denoted by $L_1$ and $L_2$, respectively. 
With these notations, let $\gamma$ be the map defined in \eqref{eqddvcmCsghH}.
On the set $H:=J(D')\cup H_k$, we define a quasiordering $\nu$ according to \eqref{eqnghmnNz}. This means that 
\begin{equation}
\begin{aligned}
\nu=\preogen\bigl(\nu_1\cup \nu_2\cup\{\pair{p}{p_1}, \pair{p_1}{p}, &\pair{p}{p_2}, \pair{p_2}{p},\dots,
\cr
&\pair{p}{p_{k+3}},\pair{p_{k+3}}{p}\}\bigr).
\end{aligned}
\label{eqdhbzGrTb}
\end{equation}
We conclude from  Lemma~\ref{lemmaHDqcol} that 
\begin{equation}
\text{$\gamma\colon \Prime(L)\to \ppair H\nu$ is a quasi-coloring.}
\label{eqtxthcGnG}
\end{equation}
Let $\delta\colon \ppair H\nu\to \ppair{J(D)}{\leq}$ be the (retraction) map defined by
\[
\delta(x)=\begin{cases}
  x,&\text{if }x\in J(D')\cup\set q=J(D),\cr
  p,&\text{if }x\in \set{p_1,p_2,\dots, p_{k+3}}.
\end{cases}
\]
Observe that if $\pair xy\in\nu_1\cup \nu_2$ 
or $x,y\in\set{p,p_1,\dots,p_{k+3}}$, then $\delta(x)\leq \delta(y)$. Hence, the set generating $\nu$ in \eqref{eqdhbzGrTb} is a subset of  $\dker(\delta)$, which implies that $\nu\subseteq \dker(\delta)$  since $\dker(\delta)$ is a quasiordering.
The inclusion just obtained means that $\delta$ is a homomorphism.
In order to verify the converse inclusion, $\dker(\delta)\subseteq \nu$, assume that $x\neq y$ and $\pair xy\in \dker(\delta)$. There are four cases.

First, assume that $x,y\in J(D')$. Then $\pair xy\in \dker(\delta)$ gives that $x=\delta(x)\leq \delta(y)=y$ in $J(D)$. But $J(D')$ is a subposet of $J(D)$, whereby $\pair x y\in\nu_1\subseteq \nu$, as required. 

Second, assume that  $\set{x,y}\cap J(D')=\emptyset$. Then $x,y\in\set{p_1,\dots,p_{k+3},q}$. Since $x\neq y$ and $\delta(x)\leq \delta(y)$, we have that $x,y\in\set{p_1,\dots,p_{k+3}}$, whereby the required containment $\pair xy\in\nu$ is clear by \eqref{eqdhbzGrTb}.

Third, assume that $x\in J(D')$ but $y\notin J(D')$.
If $y\in\set{p_1,\dots,p_{k+3}}$, then the required $\pair xy\in\nu$ follows from  $\pair xp\in \nu_1\subseteq \nu$ and $\pair py\in\nu$. Otherwise, $y=q$, and  $x=\delta(x)\leq \delta(q)=q$ gives that $\pair xe\in\nu_1\subseteq \nu$ or   $\pair xf\in\nu_1\subseteq \nu$. Since  $\pair eq,\pair fq\in\nu_2\subseteq\nu$, the required $\pair xy\in\nu$ follows by transitivity.

Fourth, assume that $x\not\in J(D')$ but $y\in J(D')$. Since $\delta(x)\in\set{p,q}$ and $\delta(x)\leq \delta(y)=y\in J(D')$, the only possibility is that $x\in\set{p_1,\dots, p_{k+3}}$ and $y=p$, which clearly yields the required $\pair xy\in\nu$.  Therefore, $\dker(\delta)\subseteq \nu$.  Thus, it follows from \eqref{eqtxthcGnG} and Lemma~\ref{lMa:wi6p} that the map
\[\whg=\delta\circ \gamma\colon \Prime(L)\to \ppair{J(D)}{\leq}, 
\qquad \text{defined by }\inr\mapsto \delta(\gamma(\inr)),
\]
is a coloring; this coloring is the same what the labeling in Figure~\ref{figexone} suggests.
By Lemma~\ref{lemmawHGhwhg}, the map $\mu\colon \ppair{J(\Con(L))}{\leq}\to \ppair{J(D)}{\leq}$ described in the lemma is an order isomorphism.
By the well-known structure theorem of finite distributive lattices, $\mu$ extends to a unique lattice isomorphism $\phi\colon \Con(L)\to D$. 
We claim that
\begin{equation}
\parbox{11cm}{an element $x$ of $D$ belongs to $\phi(\Princ(L))$ if and only if there is a chain $u_0\prec u_1\prec \dots \prec u_n$ in $L$ such that 
$x=\whg([u_0,u_1])\vee \dots \vee \whg([u_{n-1},u_n])$.}
\label{eqpbxxmcnhGbQl}
\end{equation}
In order to see this, assume that there is such a chain. Then 
\begin{equation}
\begin{aligned}
\phi(\con(u_0,u_n))&= \phi(\con(u_0,u_1)\vee\dots\vee \con(u_{n-1},u_n)) \cr
&=
\phi(\con(u_0,u_1))\vee\dots\vee \phi(\con(u_{n-1},u_n)) \cr
&=
\mu(\con(u_0,u_1))\vee\dots\vee \mu(\con(u_{n-1},u_n)) =\cr
&=\whg([u_0,u_1])\vee \dots \vee \whg([u_{n-1},u_n])=x,
\end{aligned}
\label{eqhdnGrTzmH}
\end{equation}
whereby $x\in \phi(\Princ(L))$. Conversely, assume that $x\in \phi(\Princ(L))$, that is, $x=\phi(\con(a,b))$ for some $a\leq b\in L$. Pick a maximal chain $a=u_0\prec u_1\prec \dots \prec u_n=b$ in the interval $[a,b]$, then \eqref{eqhdnGrTzmH} shows that $x$ is of the required form. This proves the validity of \eqref{eqpbxxmcnhGbQl}. 

We say that a  \emph{\eqref{eqpbxxmcnhGbQl}-chain} $u_0\prec u_1\prec \dots \prec u_n$ \emph{produces} $x$ if the equality in  \eqref{eqpbxxmcnhGbQl} holds. In order to show that $Q=\phi(\Princ(L))$, we need to show that $x\in D$ is produced by a \eqref{eqpbxxmcnhGbQl}-chain iff $x\in Q$. 
It suffices to consider join-reducible elements  and chains of length at least two, because chains of length 1 produce join-irreducible element that are necessarily in $Q$. 
First, assume that $x\in Q$. If $x\in Q'=Q\cap\ideal p$, then the choice of $D'$ and the validity of  \eqref{eqpbxxmcnhGbQl} for $D'$ and $L'$ yield a  \eqref{eqpbxxmcnhGbQl}-chain producing $x$.
Otherwise, $x$ is of the form $x=a_i\vee q$, and we can find a length two chain in $S_k\subseteq L$ that produces $x$.

Second, assume that $x$ is produced by a \eqref{eqpbxxmcnhGbQl}-chain $W$ of length at least 2. If $W$ has a $p$-colored edge, then $\filter p\subseteq Q$ implies that $x\in Q$. Hence, we can assume that no edge of $W$ is colored by $p$.  If $W$ is a chain in $L'$, then the choice of $L'$ guarantees that $x\in Q'\subseteq Q$. In $S_k$, any two edges of distinct colors not in $\set{p,q}$  are separated by a $p$-colored or $q$-colored edge. Hence, if $W$ is a chain in $S_k$, then $x$ is one of the elements   $e\vee q=q$,  $f\vee q=q$, and $a_i\vee q$ for $i=1,\dots k+3$, and these elements belong to $Q$. We are left with the case where $W$ is neither in $L'$, nor in $S_k$. Since $L$ is a Hall-Dilworth gluing of $L'$ and $S_k$, it follows that $W$ has an edge $[u_{j-1},u_j]$ such that $u_{j-1}\in L'\setminus \cast$ but $u_j\in \cast$.
In Figure~\ref{figexone}, $[u_{j-1},u_j]$ is one of the dashed lines.  By Theorem~\ref{thmggBGrs}
\eqref{thmggBGrscd}, $\whg([u_{j-1},u_j])=p$, but 
we have assumed that $W$ cannot have such an edge.
Hence, $x\in Q$ for every $W$.  Consequently, $Q=\phi(\Princ(L))$. 
This completes the proof of the  \eqref{thmmainc} $\Rightarrow$  \eqref{thmmaina} part of Theorem~\ref{thmmain}.
\end{proof}

\section{A new approach to Gr\"atzer's Theorem~\ref{thmggBGrs}}\label{secnewproof}
Our approach includes a lot of ingredients from  Gr\"atzer~\cite{ggwith1}.

\begin{proof}[Proof of the implication \textup{\ref{thmggBGrs}\eqref{thmggBGrsb} $\Rightarrow$ \ref{thmggBGrs}\eqref{thmggBGrsc}}]
Let $Q$ and $D$ be as in Theorem~\ref{thmggBGrs}\eqref{thmggBGrsb}; see Figure~\ref{figkTr}, where $Q$ consists of the grey-filled elements. 
The largest elements of $D$ will be denoted by $\pe$, it belongs to $J(D)$.
Let $\tripl C\labc D$ be a $J(D)$-labeled chain representing $Q$. We need to find an lattice $L$ and an isomorphism $\phi\colon\Con(L)\to D$ that satisfy the requirements of  \ref{thmggBGrs}\eqref{thmggBGrsc}.

\begin{figure}[ht] 
\centerline
{\includegraphics[scale=1.0]{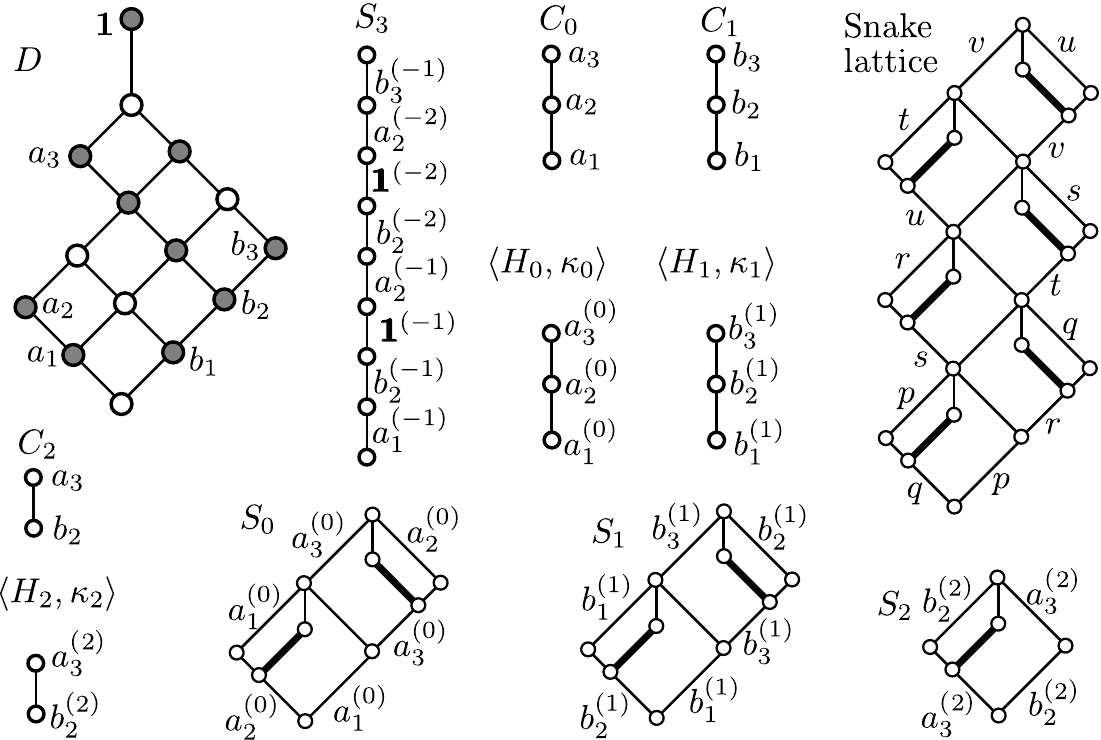}}
\caption{{An example for $Q\subseteq D$ and the first steps towards its representation}}
\label{figkTr}
\end{figure}

The ordering of $J(D)$ will often be denoted by $\kappa'$. Take a list $\ppair{C_0}{\kappa'_0}$, \dots, $\ppair{C_{t-1}}{\kappa'_{t-1}}$ of chains in $J(D)$; here $\kappa'_i$ denotes the restriction  of $\kappa'$ to $C_i$, for $i<t$. 
Assume that this list of chains is taken so that 
$J(D)\setminus\set{1_D}=\bigcup_{i<t} C_i$, 
and 
\begin{equation}
\preogen\bigl((J(D)\times\set {\pe})\cup \bigcup_{i<t} \kappa'_i\bigr)=\kappa'    
,\quad\text{that is, }
(J(D)\times\set {\pe})\vee \bigvee_{i<t} \kappa'_i=\kappa'
\label{eqchmBkappa}
\end{equation}
in $\Quo(J(D))$.
Although we can always take the list of all chains $\set{a,b}$ with $a\prec_{J(D)} b$, 
we get a much smaller lattice $L$ by selecting 
fewer chains. For $D$ given in Figure~\ref{figkTr}, we can let $t=3$, $C_0=\set{a_1<a_2<a_3}$, $C_1=\set{b_1<b_2<b_3}$, and $C_2=\set{b_2<a_3}$.
For each of the $C_i=\set{x_1<x_2<\dots<x_{m_i}}=\ppair{C_i}{\kappa'_i}$ such that $m_i>1$,
let $H_i=\set{x_1^{(i)}<x_2^{(i)}<\dots<x_{m_i}^{(i)}}=\ppair{H_i}{\kappa_i}$ be an alter ego of $C_i$; see Figure~\ref{figkTr} again. 
Each of the $\ppair{H_i}{\kappa_i}$, for $i<t$, determines a \emph{snake lattice}, which is  obtained by gluing copies of $K(\balpha,\bbeta)$ such that there is a coloring from the set of prime intervals of the snake lattice onto  $\ppair{H_i}{\kappa_i}$. For example, if 
$\ppair{H_i}{\kappa_i}$ had been  $\set{p<q<r<s<t<u<v}$, then the snake lattice would have been the one given on the right of Figure~\ref{figkTr}.
For our example, this snake lattice is not needed; what we need for our $D$ is $S_i$ and the coloring 
$\sigma_i\colon\Prime(S_i)\to \ppair{H_i}{\kappa_i}$, indicated by labels in the figure, for $i<t$. 
(By space considerations, not all edges are labeled.)
The purpose of the alter egos is to make our chains $H_i$  pairwise disjoint.
If $m_i=1$, then $S_i$ is the two-element lattice and the coloring $\sigma_i$ is the unique map from the singleton $\Prime(S_i)$ to  $\ppair{H_i}{\kappa_i} :=\ppair{\set{x_1^{i}}}{\kappa_i}$, 
where $\kappa_i$ is the only ordering on the singleton set $H_i$.

\begin{figure}[ht] 
\centerline
{\includegraphics[scale=1.0]{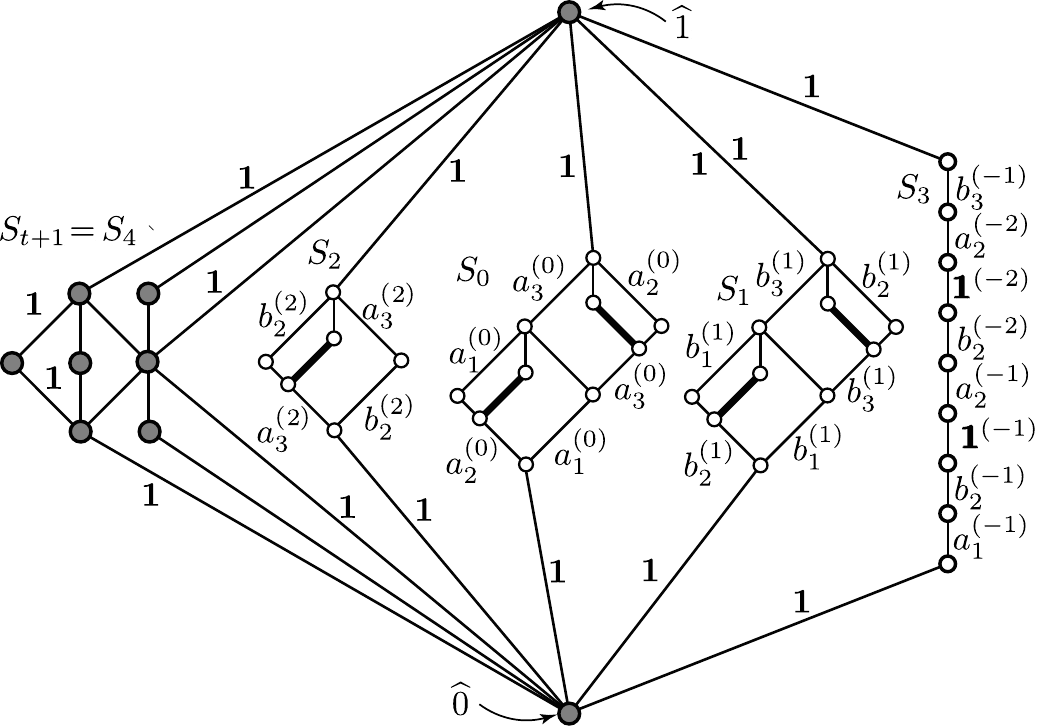}}
\caption{{The frame lattice $F$ for the example given in Figure~\ref{figkTr}}}
\label{figHmwTs}
\end{figure}

Next, we turn $\tripl C\labc D$ to a colored lattice $S_{t}$ as follows. Before its formal definition, note that in case of our example, a possible $S_t=S_3$ is given in Figure~\ref{figkTr}. 
As a lattice, $S_{t}:=C$. 
For $x\in J(D)$, let $h(x)=|\set{\inp\in\Prime(C): \labc(\inp)=x}|$; note that $h(x)\geq 1$. Let 
\[H_{t}:=\bigcup_{x\in J(D)}\set{x^{(-1)}, \dots, x^{(-h(x))}}\quad\text{ and }\quad \kappa_{t}:=\set{\pair yy: y\in H_{t}};
\]
then $\ppair{H_{t}}{\kappa_{t}}$ is an antichain, which is not given in the figure. Define the map
\begin{equation}
\parbox{10.2cm}
{$\sigma_{t}\colon\Prime(S_{t})\to H_{t}$  by the rule $\inp\mapsto x^{(-i)}$ iff $\labc(\inp)=x$ and, counting from below, $\inp$ is the $i$-th edge of $C$ labeled by $x$.}
\label{eqpbxzBfgjTkNqh}
\end{equation}
Less formally, we make the labels of $S_{t}=C$ pairwise distinct by using negative superscripts; 
see Figure~\ref{figkTr}. These new labels form an antichain $H_{t}$, and the new labeling $\sigma_{t}$ becomes a coloring.

The colored lattices $S_i$, $i\leq t$, with their colorings $\sigma_i\colon \Prime(S_i)\to \ppair{H_i}{\kappa_i}$ will be referred to under the common name \emph{branches}.
So the $i$-th branch is a snake lattice or the two-element lattice for $i<t$, and it is  $S_{t}$ with the coloring given in \eqref{eqpbxzBfgjTkNqh} for $i=t$.

\begin{figure}[ht] 
\centerline
{\includegraphics[scale=1.0]{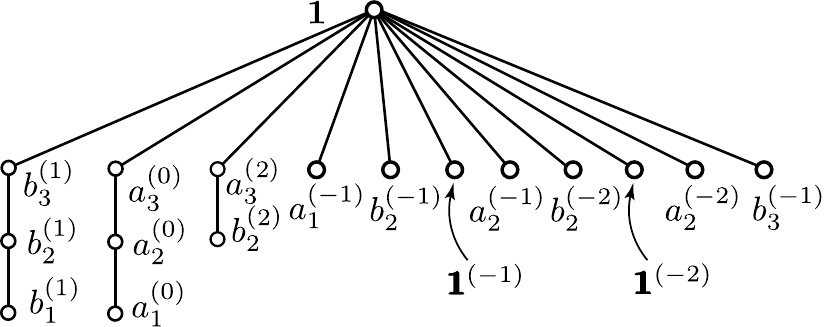}}
\caption{{$\ppair H\kappa$}}
\label{figwTQfl}
\end{figure}

Next, let  $\ppair{H_{t+1}}{\kappa_{t+1}}$
be the one-element ordered set $H_{t+1}=\set{\pe}$, and let $S_{t+1}$ be an arbitrary simple lattice with $|S_{t+1}|\geq 3$. In Figure~\ref{figHmwTs},  $S_{t+1}$ consists of the elements given by a bit larger and  grey-filled circles. We have chosen this simple lattice because it is easy to draw.  
The zero and unit of $S_{t+1}$ will be denoted by $\latn$ and $\late$, respectively.
The unique map  $\Prime(S_{t+1})\to \ppair{H_{t+1}}{\kappa_{t+1}}$ will be  denoted  by $\kappa_{t+1}$; it is a coloring.  All the lattices and ordered sets mentioned so far in this section are assumed to be pairwise disjoint. Let $F$ be the lattice we obtain from $S_{t+1}$ by inserting all the $S_i$ for $i\leq t$ as intervals such that for every $i,j\in\set{0,\dots,t}$, $x_i\in S_i$, $x_j\in S_j$ and $y_{t+1}\in S_{t+1}\setminus\set{\latn,\late}$, we have that $x_i\vee y_{t+1}=\late$, $x_i\wedge y_{t+1}=\latn$, and if $i\neq j$, then we also have that  $x_i\vee x_j=\late$ and  
$x_i\wedge x_j=\latn$; see Figure~\ref{figHmwTs}, which gives $F$ for our example of $Q\subseteq D$ given in Figure~\ref{figkTr}. Following Gr\"atzer~\cite{ggwith1}, we call $F$ the \emph{frame} or the \emph{frame lattice} associated with $Q\subseteq D$, but note that it depends also on the list of chains and the choice of $S_{t+1}$. The simplicity of $S_{t+1}$ guarantees that $F$ is a $\set{\latn,\late}$-separating lattice. In order to see this, let $x\in L\setminus\set{\latn,\late}$. If $x\in S_{t+1}$, then $\con(\latn,x)=\con(x,\late)=1_{\Con(F)}$ by the simplicity of $F$. If $x\notin S_{t+1}$, then $x$ has a complement $y$ in $S_{t+1}$, and 
$\con(\latn,x)=\con(x,\late)=1_{\Con(F)}$ since the same holds for $y$. 
Let
\begin{equation}
\begin{aligned}
H:=\bigcup_{i<t+2}H_i,\quad
\kappa:=(H\times \set 1)\cup \bigcup_{i<t+2}\kappa_i,\quad
\text{ and define }\cr
\sigma(\inp)=
\begin{cases}
\sigma_i(\inp),&\text{if }\inp\in\Prime(S_0)\cup\dots\cup\Prime(S_t),\cr
\pe,&\text{otherwise;}
\end{cases}
\end{aligned}
\label{eqksRzgThmb}
\end{equation}
in this way, we have defined a map
$\sigma\colon \Prime(F)\to \ppair H\kappa$.
In case of our example, $\sigma$ and $\ppair H\kappa$ are given by  Figures~\ref{figHmwTs} and \ref{figwTQfl}, respectively. Clearly, $\kappa$ is an ordering.
Using that $F$ is $\set{\latn,\late}$-separating and arguing similarly to  Gr\"atzer~\cite{ggwith1}, it is easy to see that $\sigma$ is a coloring.

\begin{figure}[ht] 
\centerline
{\includegraphics[scale=1.0]{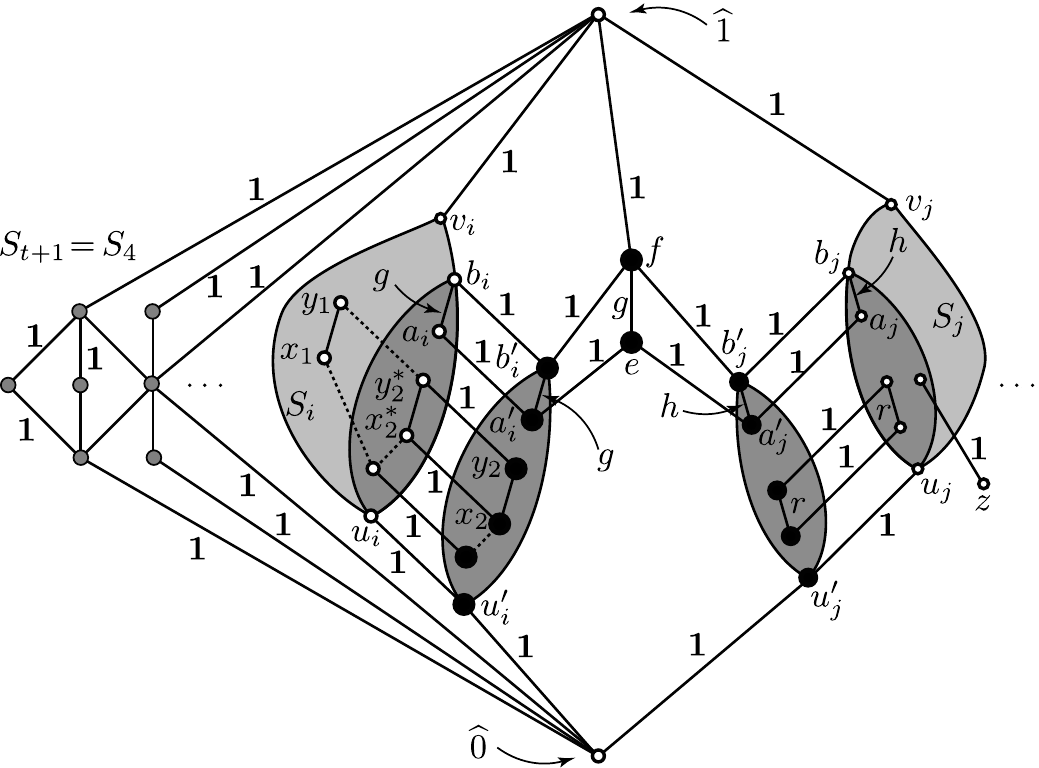}}
\caption{{Equalizing the colors $g$ and $h$}}
\label{figsizQmB}
\end{figure}

%
%

Each element of $J(D)\setminus\set{\pe}$ has at least one alter ego in $H$, but generally it has  many alter egos; they differ only in their notations and they belong to distinct branches. Note that $\pe$ also has alter egos, usually many alter egos, in $S_t$. It is neither necessary, nor forbidden that $\pe$ has alter egos in $S_0\cup\dots\cup S_{t-1}$. 
Next, let 
\begin{equation}
\begin{aligned}
\erel =\{\pair {g_0}{h_0}, 
\pair {h_0}{g_0}, 
&\pair {g_1}{h_1},
\pair {h_1}{g_1},\dots, \cr
&\pair { g_{m-1} }{h_{m-1}}, 
\pair { h_{m-1} }{g_{m-1}}   \} \subseteq H^2 
\end{aligned}
\end{equation}
be a symmetric relation such that the equivalence relation generated by $\erel$ is the least equivalence on $H$ that collapses every element with all of its alter egos. Every ``original color'' $x$ (that is, every  $x\in J(D)$) has an alter ego in $S_t$ and also in   some of the $S_{i_0}$, $i_0<t$. Since $x$ has only one alter ego in $S_{i_0}$, we can assume that
\begin{equation}
\parbox{9.5cm}{for every $\pair{g_\ell}{h_\ell}\in\erel$, there are \emph{distinct} branches $S_i$ and $S_j$ such that $S_i$ contains an edge $\inp_i$ with $\sigma(\inp_i)=\sigma_i(\inp_i)=g_\ell$ and $S_j$ contains an edge $\inp_j$ with $\sigma(\inp_j)=\sigma_j(\inp_j)=h_\ell$.}
\label{eqpbxhHbgbrghH}
\end{equation} 
Note that the smaller the $\erel$ is, the smaller the lattice $L$ will be. Since $\erel$ is symmetric, the equivalence it generates is $\preogen(\erel)$. 
Let
\begin{equation}\eta:=\preogen(\kappa\cup \erel);
\label{eqsgtZhr}
\end{equation} 
we claim that for every $x,y\in J(D)$,
\begin{equation}
\parbox{8cm}{$x\leq y$ in $J(D)$ iff there exists alter egos $x'$ and $y'$ of $x$ and $y$, respectively, such that $x'\leq_\eta y'$.}
\label{eqpbxhGtbPbT}
\end{equation}
In order to see this, assume that $x\leq y$ in $J(D)$, that is, $x\leq_{\kappa'} y$.  We can assume that $y\neq \pe$, because otherwise
$\pair x y\in  \kappa$ by \eqref{eqksRzgThmb}, whereby \eqref{eqsgtZhr} gives that 
 $\pair x y\in \eta$. By \eqref{eqchmBkappa}, there is a sequence $x=z_0, z_1,\dots, z_n=y$ in $J(D)\setminus\set\pe$ such that $\pair{z_{j-1}}{z_j}\in \bigcup_{i<t}\kappa'_i$ for all $j<n$. Denoting the corresponding alter egos by $z_j'$, we obtain by  \eqref{eqksRzgThmb} and \eqref{eqsgtZhr} that $\pair{z'_{j-1}}{z'_j}\in \bigcup_{i<t}\kappa_i\subseteq \kappa\subseteq \eta$ for all $j$, whereby $\pair{x'}{y'}=\pair{z'_0}{z'_n}\in\eta$ by transitivity. Conversely, assume that 
$x'\leq_\eta y'$ for alter egos of $x$ and $y$, respectively. 
Again, we can assume that $y\neq\pe$. It suffices to deal with the particular case $x'\leq_{\kappa_i} y'$, because the  case 
$x'\leq_{\epsilon} y'$ causes no problem and the general case follows from the particular one 
by \eqref{eqksRzgThmb}, \eqref{eqsgtZhr}, and  transitivity. But   $x'\leq_{\kappa_i} y'$ means that $x$ and $y$ belong to the same chain $C_i$ and $x\leq y$ in this chain. Hence, $x\leq y$ in $J(D)$, as required. Therefore, \eqref{eqpbxhGtbPbT} holds.

Next, we explain where the rest of the proof and that of the construction go. 
Let $\delta\colon \ppair H\eta\to \ppair{J(D)}{\leq}$, defined by $\delta(x)=y$ iff $x$ is an alter ego of $y$. Assume that we can find a lattice $L$ and a map $\gamma$ such that 
\begin{equation}
\gamma\colon\Prime(L)\to \ppair H\eta \text{ is a coloring.}
\label{eqtxtCsPdzHtMmM}
\end{equation}
Then, since $\dker{\delta}=\eta$ by  \eqref{eqpbxhGtbPbT} and  $\ppair{J(D)}{\leq}$ is an ordered set, not just a quasiordered one, it will follow\footnote{Note that the equality $\dker{\delta}=\eta$ implies that $\delta$ is a homomorphism.} from Lemma~\ref{lMa:wi6p} that, with the notation $\what\gamma:= \delta\circ\gamma$, the map
$\what\gamma\colon \Prime(L) \to \ppair{J(D)}{\leq}$ is a coloring. In the next step, it will turn out by Lemma~\ref{lemmawHGhwhg} that 
\begin{equation}
\parbox{7cm}
{$\mu\colon \ppair{J(\Con(L))}{\leq}  \to \ppair{J(D)}{\leq}$, defined by   $\con(\inp)\mapsto \what\gamma(\inp)$, in an order isomorphism.}
\end{equation}
Furthermore, \eqref{eqpbxxmcnhGbQl} will be valid for $L$ by the same reason as in Section~\ref{sectdrve}. At present, by the choice of $S_t$ and the definition of $F$, 
\begin{equation}
\parbox{9.5cm}{the elements $x\in D$ that are of the form described in  \eqref{eqpbxxmcnhGbQl}, with $F$ instead of $L$, are exactly the elements of $Q$.}
\label{eqpbxzhTrxTlSM}
\end{equation}
For $\ell\in\set{0,1,\dots, m}$, let $\erel_\ell=\set{\pair {g_j}{h_j}: j<\ell}\cup \set{\pair {h_j}{g_j}: j<\ell}$, and let 
\begin{equation}
\eta_\ell=\preogen(\kappa\cup \erel_\ell).
\label{eqGrtdK}
\end{equation} 
By \eqref{eqsgtZhr},  $\eta_0=\kappa$, $\erel_m=\erel$, and $\eta_{m}=\eta$. By induction, we intend to find lattices $L_0=F$, $L_1$, \dots , $L_m$ and quasi-colorings 
\begin{equation}
\parbox{7.5cm}
{$\gamma_0=\sigma\colon\Prime(L_0)\to \ppair{H}{\eta_0}$ and, for $\ell$ in $\set{1,\dots,m}$, $\gamma_\ell\colon\Prime(L_\ell)\to \ppair{H}{\eta_\ell}$ so that the elements described in \eqref{eqpbxxmcnhGbQl} remain the same,}
\label{eqmZrRschHdW}
\end{equation}
that is, \eqref{eqpbxzhTrxTlSM} remains valid. 
Note that $\gamma_{\ell}$ will extend $\gamma_{\ell-1}$, for $\ell\in\set{1,\dots,m}$.
Since $\gamma_0=\sigma$ and $L_0:=F$ satisfy the requirements, it suffices to deal with the transition from $L_{\ell -1}$ to $L_\ell $, for $1\leq \ell \leq m$. 

So we assume that 
$\gamma_{\ell -1}\colon \Prime(L_{\ell -1})\to \ppair{H}{\eta_{\ell -1}}$ satisfies the requirements formulated in  \eqref{eqmZrRschHdW}. 
In order to ease the notation in Figure~\ref{figsizQmB}, we denote $\pair{g_{\ell -1}}{h_{\ell -1}}$ by $\pair gh$. Then, as it is clear from \eqref{eqGrtdK}, $\eta_\ell =\preogen(\eta_{\ell -1}\cup\set{\pair{g}{h}}\cup\set{\pair{h}{g}})$. Hence, we shall add an ``equalizing flag'' $W$ to $L_{\ell-1}$ such that this flag forces that the congruence generated by a $g$-colored edge 
be equal to the congruence generated by an $h$-colored edge. The term ``flag'' and its usage is taken from Gr\"atzer~\cite{ggwith1}. Apart from terminological differences, the argument about our flag  is the same\footnote{Our $[u_i,b_i]$ in Figure~\ref{figsizQmB} is a chain in Gr\"atzer~\cite{ggwith1}, but this fact is not exploited there.} as that in Gr\"atzer~\cite{ggwith1}. By \eqref{eqpbxhHbgbrghH}, there are distinct $i,j\in\set{0,1,\dots, t}$
such that we can pick the $g$-colored edge and the $h$-colored edge mentioned above 
from branches $S_i$ and $S_j$, respectively; see Figure~\ref{figsizQmB}. 
Since the role of $g$ and $h$ is symmetric, we can assume that $i<j$. 

In Figure~\ref{figsizQmB}, the flag consists of the large black-filled elements. In order to describe the flag more precisely, let $S_i=[u_i,v_i]$ in $L_{\ell-1}$, and let $[a_i,b_i]\in\Prime(S_j)$ with $\gamma_{\ell-1}([a_i,b_i])=g$. Take the direct product of the dark-grey interval $[u_i,b_i]$ and the two-element chain $\chc 2$; this is the dark-grey interval $[u_i',b_i]$ in the figure. 
Then for every $x\in [u_i,b_i]$ there corresponds a unique element $x'\in [u'_i,b'_i]$; namely, we obtain $x'$ from $x$ by changing the ``$\chc 2$-coordinate'' of $x$ from $1_{\chc 2}$ to 
$0_{\chc 2}$. 
Then form the Hall--Dilworth gluing of the direct product and $S_i=[u_i,v_i]$ to obtain the interval $[u_i',v_i]$. Then do exactly the same with $j$ instead of $i$; see on the right of Figure~\ref{figsizQmB}.  Finally, add two more elements, $e$ as $a_i'\vee a_j'$ and $f$ as $b_i'\vee b_j'$, as shown in the figure. The lattice we obtain is $L_\ell$.  Note that Figure~\ref{figsizQmB} contains only a part of $L_\ell$; there are more branches in general (indicated by three dots in the figure) and there can be earlier flags with many additional element; one of these elements is indicated by $z$ on the right of the figure. 
We extend $\gamma_{\ell-1}$ to a map $\gamma_\ell\colon \Prime(L_\ell)\to \ppair H{\eta_\ell}$ as indicated by the figure. In particular, if
$[x,y]\in\Prime([u_i,b_i])\cup \Prime([u_j,b_j])$, then $\gamma_\ell([x',y']):= \gamma_{\ell-1}([x,y])$. We let $\gamma_\ell([e,f]):=g$ to make the definition of $\gamma_\ell$ unique\footnote{Remember, $i<j$ and the $g$-colored edge $[a_i,b_i]$ is in $\Prime(S_i)$.}, but note that $\gamma_\ell([e,f]):=h$ would also work.  
Clearly, $L_\ell$ satisfies 
\ref{thmggBGrs}\eqref{thmggBGrscb} and 
\ref{thmggBGrs}\eqref{thmggBGrscd} since so does $L_{\ell-1}$ by the induction hypothesis.

Therefore, as indicated earlier, it suffices to show that $\gamma_\ell$ is a quasi-coloring. But now this is almost trivial by the following reasons.

First, whenever we have a quasi-coloring of a lattice $U$, then it is straightforward to extend it to $U\times \chc2$: the edges $[x,y]$ and $[x',y']$ have the same color while all the $[x',x]$ edges have the same additional color. 
Since $L_{\ell-1}$ and $L_{\ell}$ are $\set{\latn,\late}$-separating, now the $[x',x]$ edges are $\pe$-colored. Apart from $e$ and $f$, which are so much separated from the rest of $L_\ell$ that they cannot cause any difficulty, we obtain the flag by two applications of the Hall--Dilworth gluing construction. Hence, the argument given for Lemma~\ref{lemmaHDqcol} works here with few and straightforward changes. Only the most important  changes and cases are discussed here; namely, the following two.

First, assume that $[x_1,y_1]\in \Prime(L_{\ell-1})$, 
$[x_2,y_2]\in \Prime(L_{\ell})\setminus \Prime(L_{\ell-1})$, and  $[x_1,y_1]\ppdn [x_2,y_2]$; see Figure~\ref{figsizQmB}. We need to show that $\gamma_\ell([x_1,y_1])\geq_{\eta_\ell} \gamma_\ell([x_2,y_2])$. 
The zigzag structure of the flag implies that $\set{x_2,y_2}$ is disjoint from $\set{e,f}$, and it follows that  
$\set{x_2,y_2}\subseteq [u_i',b_i']$  or $\set{x_2,y_2}\subseteq [u_j',b_j']$. So we can assume that 
$\set{x_2,y_2}\subseteq [u_i',b_i']$. Using that we have a Hall--Dilworth gluing (in the filter $\filter{u_i'}$ of $L_\ell$) and $[u_i',b_i]\cong [u_i,b_i]\times \chc2$, we obtain a unique $[x_2^\ast,y_2^\ast]\in\Prime([u_i,b_i])$ such that $(x_2^\ast)'=x_2$ and $(y_2^\ast)'=y_2$. Since $\gamma_{\ell-1}$ is a quasi-coloring by the induction hypothesis, we obtain that 
\begin{align*}
\gamma_\ell([x_1,y_1]) = \gamma_{\ell-1}([x_1,y_1])
\geq_{\eta_{\ell-1}}\gamma_{\ell-1}([x_2^\ast,y_2^\ast])= \gamma_{\ell}([x_2,y_2].
\end{align*}
Since $\eta_{\ell-1}\subseteq\eta_\ell$, this implies the required $\gamma_\ell([x_1,y_1])\geq_{\eta_\ell} \gamma_\ell([x_2,y_2])$. 

Second, assume that $[x_1,y_1], [x_2,y_2]\in \Prime(L_{\ell})\setminus \Prime(L_{\ell-1})$, and  $[x_1,y_1]\ppdn [x_2,y_2]$. If none of 
$\gamma_\ell([x_1,y_1])$ and $\gamma_\ell([x_2,y_2])$ is $\pe$, then 
 $\set{\gamma_\ell([x_1,y_1]), \gamma_\ell([x_2,y_2])}\subseteq\set{g,h}$, and $\pair{\gamma_\ell([x_1,y_1])}{\gamma_\ell([x_2,y_2])}\in \erel_\ell\subseteq\eta_\ell^{-1}$, as required. Otherwise, both $\gamma_\ell([x_1,y_1])$ and $\gamma_\ell([x_2,y_2])$ equal $\pe$, and we are ready by reflexivity.
\end{proof}

%

%
%
%
%
%
%
%
%
%
%

\section{Taking care of $\Aut(L)$}\label{sectlast}
A lattice $M$ is \emph{automorphism-rigid} if $|\Aut(M)|=1$. It is well-known from several sources that 
\begin{equation}
\parbox{9cm}{there exists an infinite set
$\set{M_1,M_2,M_3,M_4,\dots}$
 of pairwise non-isomorphic, automorphism-rigid finite lattices;}
\label{eqpbxHtzcfrZmb}
\end{equation}
see, for example,
 Cz\'edli~\cite[Lemma 2.8]{czgprincout}, 
Cz\'edli and Mar\'oti~\cite{czgmaroti},
Freese~\cite{freese}, Gr\"atzer \cite{gGprincIII},
Gr\"atzer and Quackenbush~\cite{grqbush}, and
Gr\"atzer and Sichler~\cite{gGsichler}. 

The easiest way\footnote{Note that there is a more involved way: if an automorphism swaps two distinct snake lattices, then it has to swap two prime intervals of $S_t$ with which these snakes are ``equalized'', but this is impossible.} to convince ourselves that the lattice $L$ constructed in the preceding sections can be chosen to be automorphism-rigid is to replace the ``thick'' prime intervals $\inp_1$, $\inp_2$, $\inp_3$, \dots in $L$ by $M_1$, $M_2$, $M_3$, \dots from \eqref{eqpbxHtzcfrZmb}, respectively, so that every edge of $M_i$ inherits the color of $\inp_i$.
This is why we have made \eqref{eqpbxshwdZhTc} a reference point. 

As a particular case of the simultaneous representability of a finite distributive non-singleton lattice $D$ and a finite group $G$ with a finite lattice $L$ in the sense that $D\cong\con(L)$ and $G\cong \Aut(L)$, it is also known that 
\begin{equation}
\parbox{7cm}{for every finite group $G$, there exists a finite simple lattice $M_0$ such that $G\cong\Aut(M_0)$.}
\label{eqpbxBrnskJ}
\end{equation}
The above-mentioned simultaneous representability is due to  Baranski\u\i~\cite{baransk} and Urquhart~\cite{urqu}; see also Gr\"atzer and Schmidt~\cite{grSch-strongindep} and Gr\"atzer and Wehrung~\cite{gGwehr} for even stronger results.
Now it is clear how to modify our constructions to complete the proofs. 

\begin{proof}[Completing the proof of Theorem~\ref{thmmyscnD}] First, do the same as in Section~\ref{secnewproof} but we have to choose $S_{t+1}$ from the list \eqref{eqpbxHtzcfrZmb}; for example, let 
$S_{t+1}=M_1$.  Then replace the ``thick'' prime intervals $\inp_1$, $\inp_2$, $\inp_3$, \dots in $L$ by $M_0$ from \eqref{eqpbxBrnskJ} and  $M_2$, $M_3$, \dots from \eqref{eqpbxHtzcfrZmb}, respectively. It follows from \eqref{eqpbxshwdZhTc} that this method works.
\end{proof}

\begin{proof}[Completing the proof of Theorem~\ref{thmmain}] In the construction described in Section~\ref{sectprbpct} and verified in Section~\ref{sectdrve}, now we shall use Theorem~\ref{thmmyscnD} rather than Theorem~\ref{thmggBGrs} to obtain $L'$. 
So let $D'=\ideal p$ as before. Since $|D'|>1$, we can choose an $L'$ that satisfies the requirements of Theorem~\ref{thmmyscnD}. In particular, $\Aut(L')\cong G$. The construction of $L'$ used some of the lattices listed in \eqref{eqpbxHtzcfrZmb};
let $i$ be the smallest subscript such that none of $M_i$ and $M_{i+1}$ was  used. 

Next, armed with $L'$,  construct $L$ as before; see Figure~\ref{figexone}. However, $L$ has two automorphisms that we do not want (and, usually, many others obtained by composition): one of these two automorphisms interchanges the two doubly irreducible elements that are the bottoms of $e$-colored edges, while the other one does the same with $f$ instead of $e$. To get rid of these unwanted automorphisms, \eqref{eqpbxshwdZhTc} allows us to replace the $e$-colored thick edge and the $f$-colored \emph{thick} edge in Figure~\ref{figexone} by $M_i$ and $M_{i+1}$, respectively. The new lattice we obtain in this way, which is also denoted by $L$ from now on, has only those automorphisms that are extensions of automorphisms of $L'$. Hence, $\Aut(L)\cong\Aut(L')\cong G$, as required.
\end{proof}

\end{document}